\documentclass[12pt
]{amsart}

\usepackage[margin=1in]{geometry}

\usepackage{amscd}
\usepackage{amsmath}
\usepackage{graphicx}
\usepackage{wasysym}
\usepackage{enumerate}
\usepackage{ifpdf}
\usepackage{amsfonts}
\usepackage{amsthm}
\usepackage[hyperfootnotes=false]{hyperref}
\hypersetup{hidelinks}
\usepackage{setspace}
\usepackage[usenames]{xcolor}
\usepackage{tikz}

\usepackage{amsrefs}

\usepackage{amsfonts}
\DeclareMathSymbol{\blacktriangle}{\mathord}{AMSa}{"4E}
\DeclareMathSymbol{\nsubseteq}{\mathrel}{AMSb}{"2A}
\DeclareMathSymbol{\lesssim}{\mathord}{AMSa}{"2E}
\DeclareMathSymbol{\gtrsim}{\mathord}{AMSa}{"26}

\title[]{Smooth equivalence of deformations of domains in complex euclidean spaces}

\author[]{Herv\'{e} Gaussier$^\dag$ \and Xianghong Gong$^{\dag\dag}$}
\address{H. Gaussier:
Univ. Grenoble Alpes, CNRS, IF, F-38000 Grenoble, France
}
\email{herve.gaussier@univ-grenoble-alpes.fr}
\address{X. Gong:
Department of Mathematics,
University of Wisconsin-Madison, Madison, WI 53706, U.S.A.}
\email{gong@math.wisc.edu}

  \thanks{$^\dag\,$Partially supported by ERC ALKAGE. $^{\dag\dag}$Partially supported by a grant from the Simons
Foundation (award number: 505027)}

\oddsidemargin=\evensidemargin

\newcommand{\dist}{\operatorname{dist}}
\newcommand{\DD}[2]{\frac{\partial #1}{\partial #2}}

\newtheorem{thm}{Theorem}[section]
\newtheorem{cor}[thm]{Corollary}
\newtheorem{prop}[thm]{Proposition}
\newtheorem{lemma}[thm]{Lemma}

\theoremstyle{definition}

\renewcommand{\th}[1]{\begin{thm}\label{#1}}
\newcommand{\eth}{\end{thm}}
\newcommand{\co}[1]{\begin{cor}\label{#1}}
\newcommand{\eco}{\end{cor}}
\renewcommand{\le}[1]{\begin{lemma}\label{#1}}
\newcommand{\ele}{\end{lemma}}
\newcommand{\pr}[1]{\begin{prop}\label{#1}}
\newcommand{\epr}{\end{prop}}

\newcommand{\ga}{\begin{gather}}
\newcommand{\ega}{\end{gather}}
\newcommand{\gan}{\begin{gather*}}
\newcommand{\egan}{\end{gather*}}
\newcommand{\al}{\begin{align}}
\newcommand{\eal}{\end{align}}
\newcommand{\aln}{\begin{align*}}
\newcommand{\ealn}{\end{align*}}
\newcommand{\eq}[1]{\begin{equation}\label{#1}}
\newcommand{\eeq}{\end{equation}}

\newcommand{\ci}{~\cite}

\newcommand{\f}[2]{\frac{#1}{#2}}

\newcommand{\D}{\mathbb{D}}
\newcommand{\cc}{{\bf C}}

\newcommand{\rr}{{\bf R}}

\newcommand{\hcc}{\hat{\bf C}}

\newcommand{\ov}{\overline}

\newcommand{\id}{\operatorname{Id}}
\newcommand{\RE}{\operatorname{Re}}
\newcommand{\IM}{\operatorname{Im}}

\newcommand{\cL}{\mathcal}

\newcommand{\all}{\alpha}

\newcommand{\gaa}{\gamma}
\newcommand{\Gaa}{\Gamma}
\newcommand{\del}{\delta}

\newcommand{\var}{\varphi}
\newcommand{\e}{\epsilon}
\newcommand{\om}{\omega}
\newcommand{\Om}{\Omega}

\newcommand{\la}{\lambda}

\newcommand{\pd}{\partial}

\newcommand{\re}[1]{(\ref{#1})}
\newcommand{\rea}[1]{$(\ref{#1})$}
\newcommand{\rl}[1]{Lemma~\ref{#1}}
\newcommand{\nrc}[1]{Corollary~\ref{#1}}
\newcommand{\rp}[1]{Proposition~\ref{#1}}
\newcommand{\rt}[1]{Theorem~\ref{#1}}

\newcommand{\rpa}[1]{Proposition~$\ref{#1}$}
\newcommand{\rta}[1]{Theorem~$\ref{#1}$}

\newcounter{pp}
\newcommand{\bpp}{\begin{list}{$\hspace{-1em}\alph{pp})$}{\usecounter{pp}}}
\newcommand{\epp}{\end{list}}

\newcounter{ppp}
\newcommand{\bppp}{\begin{list}{$\hspace{-1em}(\roman{ppp})$}{\usecounter{ppp}}}
\newcommand{\eppp}{\end{list}}

\begin{document}

\begin{abstract}
We prove that two smooth families of 2-connected domains in $\cc$ are smoothly equivalent if they are equivalent under a possibly discontinuous family of biholomorphisms. We construct, for $m \geq 3$,  two  smooth families of smoothly bounded $m$-connected domains in $\cc$, and for $n\geq2$,  two families of strictly pseudoconvex domains in $\cc^n$,   that are  equivalent under    discontinuous families of biholomorphisms but   not under any continuous family of biholomorphisms.  Finally, we give sufficient conditions for  the smooth equivalence of two smooth families of domains.
 \end{abstract}

 \maketitle


\setcounter{section}{0}
\setcounter{thm}{0}\setcounter{equation}{0}

\section{Introduction}\label{sec1}

  The main purpose of this paper is to study smoothness properties of holomorphic equivalence for  deformations of domains in $\cc^n$.
The deformation theory  was developed by Kodaira-Spencer~\ci{KS58} for
 compact complex manifolds.  In the local theory,
the study of a family of complex structures has also occurred in the work of Newlander-Nirenberg~\cite{NN57}, Nirenberg~\ci{Ni57}, and Nijenhuis-Woolf~\ci{NW63}.  Families of complex structures were further studied recently by Bertrand-Gong-Rosay~\cite{BGR14}, and the second author~\ci{Go17}.

We consider a family  $\{D^\la;0\leq \la\leq1\}$ of domains $D^\la$ in $\cc^n$. Thus  the total space $\cL D=\cup_{ t\in[0,1]}D^\la\times\{\la\}$ is a subset of $ \cc^n\times[0,1]$ and we will also denote $\{D^\la\}$ by $\cL D$.
We  will study the equivalence of $\cL D$ under a family $F:=\{F^\la\}$ of biholomorphic mappings $F^\la$ from $D^\la$ onto suitable domains in $\cc^n$ and we are interested in the regularity of the total map $F\colon\cL D\to \cc^n\times[0,1]$ defined by $F(z,\la)=F^\la(z)$.    The Riemann mapping theorem says that a simply connected domain in the complex plane is holomorphically equivalent to the unit disc when $D$ is not the whole plane. For domains in $\cc^n$ with $n>1$, there is no Riemann mapping theorem in the sense that topology of the domains plays no role for the holomorphic classification of the domains and a relevant question is if a biholomorphic mapping between two smoothly bounded domains extends smoothly up to the boundary. Fefferman's theorem~\ci{Fe74} says that such a mapping indeed extends smoothly when the domains are strictly pseudoconvex. For general domains,
the boundary extension of proper holomorphic mappings has been studied intensively and many positive results have been obtained; see for instance the survey article by Forstneri\v{c}~\cite{Fors93} and the reference in Diederich-Pinchuk~\cite{Di-Pi95} for biholomorphic extension between real-analytic non-pseudoconvex domains in complex dimension 2. The reader can also refer to a survey by Coupet-Gaussier-Sukhov~\ci{CGS08} on the Fefferman type of extension theorem for domains in  almost complex manifolds.

For the classification of families of domains under families of biholomorphic mappings, there are only very few results.    Courant proved a version
of Carath\'eodory's Riemann mapping theorem
for a sequence of Jordan domains, showing the continuous dependence of the Riemann mappings on the domains (see Tsuji~\ci{Ts59}, p.~383).  Recently Bertrand-Gong~\ci{BG14} proved that a smooth family of smoothly bounded simply connected domains in the complex plane is equivalent to the unit disc under a smooth family of biholomorphic mappings. In this paper, we will prove the following:
\pr{annulus}
Let $\cL D$ be a smooth family of smoothly bounded $2$-connected   domains $D^\la$ in $\cc$. Then there exists a smooth family
$ K
$
of diffeomorphisms $K^{\la}$
from $\ov{D^{\la}}$ onto $\ov{A^{\la}}$ with
$$
A^{\la}:=\{z \in \cc \colon \mu(\la) < |z| < 1\}
$$
such that each $K^{\la}$ is holomorphic on $D^{\la}$.
\epr
Although the smoothness of equivalence is sometimes
 considered as granted in deformation theory, the following question was raised in Bertrand-Gong~\cite{BG14}: Let $\cL D, \widetilde{\cL D}$ be two smooth families of smoothly bounded domains in $\cc^n$. Suppose that for each $\la$ there is a biholomorphic mapping $F^\la$ sending $D^\la$  onto $\tilde D^\la$. Does there exist a smooth family of biholomorphic mappings $G^\la$ from $D^\la$ onto $\tilde D^\la$ such that the total map $G\colon\ov{\cL D}\to\ov{\widetilde{\cL  D}}$ is smooth? One of the main results of this paper is the following:
\th{negative} Let $n$ be a positive integer.
There exist two smooth families $\cL D
$, $ \widetilde{\cL D}
$ of smoothly bounded domains in $\cc^n$
such that each $D^{\la}$ is biholomorphic to $\tilde D^{\la}$, while there is no continuous family $F
$
of biholomorphic mappings $F^\la$ from $D^{\la}$ onto $\tilde D^{\la}$.
\eth

Therefore, to seek a classification of a smooth family of domains under a smooth family of biholomorphisms, we need to impose additional conditions on the domains. In this paper, we consider a family of {\it rigid} domains $D^\la$. Here by a rigid domain we mean that its holomorphic automorphism group   consists of the identity map only. By a theorem of Greene-Krantz~\ci{GK84}, rigid strictly pseudoconvex domains form a dense and open set in the space of   bounded domains of $C^2$ boundary.
 It turns out  that  the   non-rigidity plays a role in the proof of \rt{negative}.
It remains open if the above-mentioned question has a positive answer when all domains are rigid.  As a first step,  we will prove the following positive result for rigid domains:
\th{boundary-}Let $\cL D$, $\widetilde{\cL D}$ be two continuous families of bounded strictly pseudoconvex domains with $C^2$ boundary. Suppose that for each $\la$, $D^\la$ is rigid and  $F^\la$ is a biholomorphic mapping from $D^\la$ to $\tilde D^\la$. Then $F$ is a continuous family of homeomorphisms $F^\la$ from $\ov {D^\la}$ onto $\ov{\tilde D^\la}$. Furthermore, the H\"{o}lder-$\f{1}{2}$ norms of $F^\la$ on $\ov{D^\la}$ have an upper bound independent of $\la$.
\eth
When $n=1$, a higher regularity result is given by \nrc{preg}. See also Courant~\cite{Co50}*{p.191}
 for related results when $D^\la$ is independent of $\la$ and $D^\la,\tilde D^\la$ are bounded by finitely many Jordan curves.  For a single biholomorphic mapping between two strictly pseudoconvex domains, which are not necessarily rigid, the result is due to Margulis~\ci{Ma71}, Vormoor~\ci{Vo73},  and  Henkin~\ci{He73}.

The paper is organized as follows. In section~\ref{sec2}, we find a normal form for a family of multi-connected domains in the complex plane. The normal form includes \rp{annulus} for the two-connected domains. Using the normal form, we also provide a complete solution with necessary and sufficient conditions ensuring the smooth dependence of biholomorphisms in a parameter, for families of rigid domains in the complex plane. In section~\ref{sec3} we construct examples of smooth families of domains in $\cc^n$,
for any $n\geq1$, that are equivalent under a discontinuous family of biholomorphisms, while there are not equivalent under any {\it continuous} family of biholomorphisms. This is the content of \rt{negative}. A common feature of the examples is that each of the smooth families of domains consists of only rigid domains except one. The solution to the one-dimensional problem and the examples   constructed in Theorem~\ref{negative} lead us to seeking sufficient conditions in higher dimensions to ensure positive results, and \rt{boundary-}, proved in section~\ref{sec4},  is our main positive result in the higher dimension.

The main consequence of \rt{boundary-} is a localization for the study of the global regularity of the family of holomorphic mappings between two families of rigid domains to a local study by employing the existing tools in several complex variables. As an application, we obtain $C^{\infty,0}$ regularity of $\{F^\la\}$ as well as $C^{\infty,\infty}$ and $C^{\om,\om}$ regularities of $\{F^\la\}$ under additional assumptions on the interior regularity  $\{F^\la\}$ with respect to all variables by employing a simplified proof by Lempert~\ci{Le81}
for Fefferman's theorem and the Lewy-Pinchuk theorem (see Corollaries~\ref{cor47} and \ref{cor48}).

\setcounter{thm}{0}\setcounter{equation}{0}

\section{One complex variable positive results}\label{sec2}

Let us first  define spaces on domains with parameter.
Let $D$ be a bounded domain in $\rr^N$. Let $0\leq\all<1$ and let $k,j$ be integers with $k\geq j\geq0$. We say that a family $u=\{u^\la\}$ of functions $u^\la$ is
in $ C^{k+\all,j}(\ov D)$ if
for each $\ell+i\leq k$ and $i\leq j$, the partial derivative $\pd_x^\ell\pd_t^iu$ is continuous and the H\"older-$\all$ ratios in the variables $x\in\ov D$ are uniformly bounded as $\la$ varies in $[0,1]$. By a family $\cL D=\{D^\la\}$ of  domains $D^\la$ of class $C^{k+\all,j}(\ov D)$ we mean that there are embeddings $\Gaa^\la$ from $\ov D$ onto $\ov{D^\la}$ so that each component of $\Gaa:=\{\Gaa^\la\}$ is a family of functions in $C^{k+\all,j}(\ov D)$ and $D$ is a bounded domain of $ C^{k+\all}$ boundary, with $k\geq1$; we call $\Gaa$ a {\it parameterization} of $\cL D$ of class $C^{k+\all,j}$. Define $\pd \cL D=\cup_\la\pd D^\la\times\{\la\}$.  When $D$ is a family of domains $\{D^\la\}$ given by embeddings $\Gaa^\la$ from $\ov D$ onto $\ov{D^\la}$  and when $\pd D\in C^{k+\all}$ and $k\geq1$,
we say that a family $\{v^\la\}$ is in $ C^{\ell+\all,j}(\pd {\cL D})$ if $v^\la$ are functions on $\pd D^\la$ and $\{v^\la\circ\Gaa^\la\}$ is of class $C^{\ell+\beta,j}(\pd D)$ for $\ell+\beta\leq k+\all$ and $\ell\geq j$. Analogously, one can defined a real analytic family $\cL D$ of bounded domains $D^\la$ with real analytic boundary, and a real analytic family $u$ of functions $u^\la$ on $\ov{D^\la}$.  Thus, we can consider spaces $C^{\om,\om}(\ov D)$ and $C^{\om,\om}(\ov{\cL D})$.
  See~\ci{BG14} for basic properties of the spaces $C^{k+\all,j}(\ov{\cL D})$. Note that the functions spaces are independent of the parameterization $\Gaa$.

We will need the following  result from Bertrand-Gong~\ci{BG14}.

\pr{dpp} Let  $j,k$ be non negative integers or $\infty$  satisfying
 $k\geq j$. Let   $0<\all<1$. Let $D$ be  a bounded domain in $\rr^2$ with $C^{k+1+\all}$ $($resp.~$C^\om)$ boundary.
Let $\cL D$ be a family of   bounded domains $\cL D^\la$ in $\rr^2$ with $\ov{\cL D} \in C^{k+1+\all,j}(\ov D)$ $($resp.
  $C^{\om,\om}(\ov D))$.
Let $u\in C^{0}(\pd \cL D)$
 be  harmonic functions $u^\la$ on ${ D^\la}$
with $\{u^\la|_{\pd D^\la}\}\in C^{k+1+\all,j}(\pd\cL D)$ $($resp.~$C^{\om,\om}(\pd\cL D))$.  Then $u\in C^{k+1+\all,j}(\ov{\cL D})$ $($resp.~$C^{\om,\om}(\ov{\cL D}))$.
\epr
We remark that there is a difference in notation. The space $C^{k+1+\all,j}(\cdot)$ is the space $\cL B(\cdot)$ in~\ci{BG14}. Although some results in this paper can be formulated for the space $\cL C^{k+1+\all,j}(\cdot)$ in~\ci{BG14}, we will not discuss them in this paper.

Before we deal with multi-connected domains, we recall here
a positive result on simply connected domains   proved in \cite{BG14}.
\pr{dpp+} Let  $j,k,\all, D, 
\cL D$ be as in \rpa{dpp}.
Suppose that $\cL D\in C^{k+1+\all,j} ($resp. $C^{\om,\om})$ and $D$  is simply connected. Then there exists a family $R$ of Riemann mappings
  $R^\la$ from $D^\la$ onto the unit disc so that  $R$ is in $C^{k+1+\all,j}(\ov{\cL D})$ $($resp.~$C^{\om,\om}(\ov{\cL D}))$.
\epr

We consider now the  $m$-connected case, with $m \geq 2$. Let us first introduce the following notation.
Let $\D_{r}(c)$ be the disc of radius $r$ centered at $c$, with $\D_{r}=\D_r(0)$.
  For $0 < r < 1$, let $A_r$ be the annulus $\D_1\setminus\ov{\D_r}$.
For $m \geq 3$ and $j=3,\dots,m$, let $\alpha_j < \beta_j < \alpha_j + 2\pi$.
Let
$$A^*_{m,r}(\all,\beta):=A^*_{m,(r_2,\dots,r_m)}((\all_3,\dots,\all_m),(\beta_3,\dots,\beta_m))
$$
be the $m$-connected domain obtained by removing
from $A(r_2)$ with $0<r_2<1$, the $(m-2)$ disjoint
arcs
$$\operatorname{Arc} (r_j,[\all_j,\beta_j]):=\{r_je^{i\theta}\colon\all_j\leq\theta\leq\beta_j\},\ 3 \leq j \leq m.
$$
 We call $A^*_{m,r}(\all,\beta)$ a {\it slit annulus}.

\vspace{2mm}
Given a family of embeddings $\Gaa^{\la}\colon\ov D\to\ov{D^{\la}}$ with
$\Gaa\in C^{1,0}(\ov D)$ and $\pd D\in C^1$, we will label the boundary of $D$
as union of disjoint closed curves $\gaa_1,\ldots, \gaa_m$. We set $\gaa_i^{\la}=\Gaa^{\la}(\gaa_i)$.
We might  adjust the order and the orientations
 of $\gaa_1,\dots, \gaa_m$ when necessary.

We first prove the following Koebe norm form for a family of $m$-connected domains, for $m \geq 2$.

\th{1dmarked} Let $j,k,\all, D,\cL D$ be as in \rpa{dpp}.
Suppose that $D$ is  $m$-connected and $\cL D\in C^{k+1+\all,j} ($resp.~$C^{\om,\om})$. 
 Let $\la \mapsto (a_1({\la}),a_2({\la}))\in\pd D^{\la}\times \pd D^{\la}$  be  a $C^j$ $($resp.~$C^\om)$ mapping
 such that $a_1({\la}),a_2({\la})$ are not in the same component of $\pd D^{\la}$.
 The following hold~$:$
   \bppp
\item  There exists
  a unique family $K$ of canonical conformal mappings $K^{\la}$ from $D^{\la}$ onto the slit annulus
  $A^*_{m,r^{\la}}(\all^\la,\beta^\la)
  $
  such that $K^{\la}(a_1({\la}))=1$, $|K^{\la}(a_2({\la}))|=r_2^{\la}$.
Moreover,   $K\in C^{k+1+\all,j}(\ov{\cL D})$ $($resp.~$ C^{\om,\om}(\ov{\cL D}))$.
  \item Suppose that
  $$
  \tilde K^{\la}\colon D^{\la}\to   A^*_{m,\tilde r^\la}(\tilde\all^\la,\tilde\beta^\la)
  $$
  satisfies $\tilde K^{\la}(\gaa_1^{\la}) = \pd\D_1$ and $\tilde K^{\la}(\gaa_2^{\la}) = \pd \D_{\tilde r^{\la}_2}$.
Then $\tilde K^{\la}=\mu^{\la} K^{\la}$ with $|\mu^{\la}|=1$.
   \item
    The mapping   $\la \mapsto r_2^\la$ and, if $m \geq 3$, the mapping
 $$
 \la \mapsto ( r_3^\la,
 \all_3^{\la},\beta_3^{\la},\dots,r_m^\la, \all_{m}^{\la},\beta_{m}^{\la})
$$
are of class $C^j([0,1])$
  $($resp.~$C^{\om}([0,1]))$.
  \item   Let $m \geq 3$. For each $\la$ and $\ell=3,\dots, m$, $a_\ell^\la:=(K^\la)^{-1}(r_i^\la e^{i\all_\ell^\la})$ and $b_\ell^\la:=(K^\la)^{-1}(r_i^\la e^{i\beta_\ell^\la})$ are points in $\gaa_\ell^\la$. Moreover, the map $\la\mapsto(a^\la,b^\la)$ is  of class  $C^j$ $($resp.~$C^\om)$ in $[0,1]$.
 \eppp
  \eth

  \begin{proof}
 $(i).$ The existence of the canonical mapping $K^{\la}$ for a fixed $\la$, given by $(i)$ is a theorem of Koebe. For its proof, see
   Ahlfors~\ci{Ah78}*{p. 255}.

 We already get
$$
K^{\la}\colon D^{\la}\to A_{m,r^\la}^*(\all^\la,\beta^\la).
$$
We now verify the smoothness with respect to $\lambda$.

We label the component that contains $a_1^{\la}$ as $\gaa_1^{\la}$ and the one that contains $a_2({\la})$ as $\gaa_2^{\la}$. We label the other components as $\gaa_3^{\la},\dots, \gaa_m^{\la}$,
depending continuously on $\la$.

For $i=1,\dots,m$, let $u_i^{\la}$ be the harmonic measure on $\pd D^{\la}$ such that $u_i^{\la}=1$ on $\gaa_i^{\la}$
  and $u_i^{\la}=0$   on $\gaa_\ell^{\la}$, for $\ell\neq i$.  By \rp{dpp}, $\{u_i^{\la}\}$ has the desired regularity.
  Let $\nu^\la$ be the outer unit normal vector of $\pd D^{\la}$. Then it follows from the maximum principle:
  $$
\left\{
\begin{array}{lllll}
\pd_{\nu^\la} u_i^{\la} & > & 0 & \text{on} & \gaa_i^{\la},\\
 & & & \\
\pd_{\nu^\la}u_i^{\la}   & < & 0 & \text{on} & \gaa_\ell^{\la},
\ \ell\neq i.
\end{array}
\right.
$$
Since $\int_{D^{\la}}\partial \overline{\partial}u_i^{\la} = 0$, it follows from the Green formula that $\int_{\pd D^{\la}}\pd_{\nu^\la}u_i^{\la}\, d\sigma=0$ for length element $d\sigma$
  on $\pd D^{\la}$.  Moreover, we have
  \begin{equation}\label{det-eq}
  \det\left(\int_{\gaa_i^{\la}}\pd_{\nu^\la}u_\ell^{\la}\, d\sigma\right)_{2 \leq i,\ell\leq m} \neq 0.
  \end{equation}
Indeed, suppose by contradiction that there exist $\la \in [0,1]$ and $(c_2^{\la},\dots,c_m^{\la}) \neq 0$ such that
  $$
  \sum_{\ell=2}^m c_\ell^{\la}\int_{\gaa_i}\pd_{\nu^\la} u_\ell^{\la}\, d\sigma=0, \quad i \geq 2.
  $$

 Let $\tilde u^{\la}=\sum_{\ell=2}^m c _\ell^{\la}u _\ell^{\la}$.
 Then we have, for every $i \geq 2$:
 $$
\displaystyle \int_{\gaa_i^{\la}} \pd_{ {\nu^\la}}\tilde u^{\la}\, d\sigma = \displaystyle \sum_{\ell=2}^mc_{\ell}^{\la} \int_{\gaa_i^{\la}} \pd_{ {\nu^\la}}u_{\ell}^{\la}d\sigma = 0
 $$
 and, from the Green formula: $\displaystyle \int_{\gaa_1^{\la}} \pd_{ {\nu^\la}}\tilde u^{\la}\, d\sigma = - \displaystyle \sum_{i=2}^m\int_{\gaa_i^{\la}} \pd_{ {\nu^\la}}\tilde u^{\la} \, d\sigma=0$.
 If $\tilde u^{\la}$ is not constant on $D^{\la}$ then, since $\tilde u^{\la}$ is constant on each $\gaa_i^{\la}$, it follows from the Hopf lemma that $\pd_{ {\nu^\la}}\tilde u^{\la}>0$ on  some $\gaa_{i_*}^{\la}$ ($i \geq 1$)
 on which $\tilde u^{\la}$ attains the maximum value. This contradicts the equality
 $\int_{\gaa_{i_*}^{\la}}\pd_{\nu^\la}\tilde u^{\la}\, d\sigma=0$.

 Hence $\tilde u^{\la}$ is constant on $D^{\la}$. Then $\tilde u^{\la}=0$ as it holds on $\gaa_1^{\la}$.
 However for every $i \geq 2$, $\tilde u^{\la}=c_i^{\la}$ on $\gaa_i^{\la}$. Hence, for every $i \geq 2$, $c_i^{\la} = 0$. This is a contradiction.

 It follows now from (\ref{det-eq}) that for every $\la \in [0,1]$ there is a unique   $(m-1)$-tuple $(c_2^{\la},\dots,c_m^{\la}) \neq 0$ such that $u^{\la}:=\sum_{i=2}^mc_i^{\la}u_i^{\la}$ satisfies
 \begin{equation}\label{mod-eq}
 \int_{\gaa_2^{\la}} \pd_{ {\nu^\la}}u^{\la}\, d\sigma = -2\pi, \quad  \int_{\gaa_i^{\la}} \pd_{ {\nu^\la}}u^{\la}\, d\sigma=0,
 \quad 3 \leq i\leq m.
 \end{equation}
 Thus $\displaystyle \int_{\gaa_1^{\la}}\pd_{ {\nu^\la}}u^{\la}\, d\sigma = 2\pi$.
 Since $u^{\la}$ is not identically zero and is constant on each $\gaa_i^{\la}$, then
by the Hopf lemma again $\pd_ {\nu^\la} u^\la>0$ on the component on which
it has maximum value and $\pd_{ {\nu^\la}}u^\la<0$ on the component on which $u^\la$ has  the minimum value. Thus
  $u^{\la}$ does not attain maximum or minimum values on $\gaa_i^{\la}$ for $3 \leq i \leq m$.
 So $u^{\la}$ attains the maximum value on one of the two curves $\gaa_1^{\la},\gaa_2^{\la}$ and the minimum value on the
 other curve.  From $\int_{\gaa_1^{\la}}\pd_{ {\nu^\la}}u^{\la}\, d\sigma = 2\pi$ and $\int_{\gaa_2^{\la}}\pd_{ {\nu^\la}}u^{\la}\, d\sigma = -2\pi$, we know that $u^{\la}$ attains the maximum value on $\gaa_1^{\la}$ and the minimum value on $\gaa_2^\la$.  Furthermore, the maximum value is $0$.

 By \rp{dpp}, as a function in $\la$, each entry in the matrix in \re{det-eq} has the same regularity as of $\pd \cL D$. Hence the unique solution $(c_2^\la,\dots,c_m^\la)$, which is obtained
 from the linear system \re{mod-eq},  also has the same regularity.
  In particular, $u^{\la}$ has the same regularity in $(z,\la)$ as $u^{\la}_2,\dots,u^{\la}_m$, given by Proposition~\ref{dpp}.

 Let $v^{\la}$ be the harmonic conjugate of $u^{\la}$. It follows from (\ref{mod-eq}):
$$
 v^{\la}(z):=\int_{a_1(\la)}^z(-\pd_{y}u^{\la}\, dx+\pd_{x}u^{\la}\, dy), \quad \mod 2\pi.
$$
Assume that $\la\mapsto a_1(\la)$ is in $C^j([0,1])$.  We can verify that
$v^{\la}$ has the same regularity as $u^{\la}$
since $k\geq j$.

Set $K^{\la}(z)=e^{u^{\la}(z)+iv^{\la}(z)}$ for $z \in D^{\la}$.
 By Rouch\'e's theorem, $K^{\la}$ is a biholomorphic mapping from $D^{\la}$ onto
 $A^*_{m,r^\la}(\all^\la,\beta^\la)$.
We have $K^{\la}(a^{\la}_1)=1$ and
$$
r_i^\la=|K^\la(z)|=e^{u^\la(z)}, \quad z\in\gaa_i^\la, \quad 1<i\leq m.
$$
  It follows that the map $\la \mapsto r^{\la}_i$ is in $C^j([0,1])$. It is clear that
  $K^{\la}$ has the same regularity as $u^{\la}$ and $v^{\la}$.

$(ii).$ The assertion is proved in~\ci{Ah78}*{p.~257}.

 $(iii).$
Fix $i$ with $2<i\leq m$ and fix $\la_0$.  We know that $v^\la \mod 2\pi$ is continuous.
 Fix the  range of $v^{\la_0}$ on $\gaa_i^{\la_0}$ in $[\all_i^{\la_0},\beta_i^{\la_0}]$.  We may assume that
 $v^\la$ is continuous and has values in $I:=(-\e+\all_i^{\la_0},\beta_i^{\la_0}+\e)$, when  $\la$ is sufficiently close to $\la_0$.  Here the length of $I$ is less than $2\pi$. Then we may take $\all_i^\la$ to be the minimum value of $v^\la$ and $\beta_i^\la$ its maximum value. We have $v^\la(\Gaa^\la(x_\la))=\all_i^\la$.
 By continuity of $v^\la(\Gaa^\la)$ and compactness of $\gaa_i$,
 we conclude $\liminf_{\la\mapsto\la_0}\all_i^\la\geq \all_i^{\la_0}$.  Suppose that $\all_i^{\la_0}=v^{\la_0}(\Gaa^{\la_0}(x_0))$. Let $\e>0$. By the continuity of $v^\la(\Gaa^\la)$, we also have
  $v^\la(\Gaa^\la(x))<\all_i^{\la_0}+\e$ for $(\la,x)$ sufficiently close to $(\la_0,x_0)$. This shows that
  $\limsup_{\la\mapsto\la_0}\all_i^\la\leq\all_i^{\la_0}$. Therefore, $\all_i^\la \mod (2\pi)$ is continuous. Analogously, $\beta_i^\la \mod (2\pi)$ are also continuous in $\la$.

The above ends the proof for $j=0$.

Suppose now that $j\geq1$.
The latter implies that
 $k+1\geq2$.
We want to show that the mapping $\la\mapsto(\all_j^\la,\beta_j^\la)$ is of class $C^j$.  Recall that the mapping $z\mapsto v^\la(z)$ is of class $C^{k+1+\all,j}$ and
$$
K^\la(z)=r_j^\la e^{iv^\la(z)}, \quad z\in\gaa_j^\la.
$$

\vspace{2mm}
Fix $j$ and $\la$. Let $K^\la(a_j^\la)=r_j^\la e^{i\all_j^\la}$. Let $2<j\leq m$. Let us first consider the case that
$\gaa_j^\la$ is real analytic near $a_j^\la$.
By the Schwarz reflection principle, we know that
 the map
$$
f : D^{\la}\ni z   \mapsto f(z)=\left(\f{K^\la(z)-K^\la(a_j^\la)}{K^\la(z)+K^\la(a_j^\la)}\right)^{1/2}
$$
is biholomorphic near $a_j^\la$.
Thus, we have
$
f'(a_j^\la)\neq0.
$
Assume that $\tilde v(s,\la):=v^\la(\gaa_j^\la(s))$ attains the minimum value at $s=s^\la_*$. Thus its derivative at $s=s^\la_*$ vanishes. Then a straightforward computation  by differentiating $f^2(\gaa_j^\la(s))$ in $s$ twice shows that
\eq{non-zero-der}
\DD{^2}{s^2}  \tilde v(s^\la_*,\la)=-\sqrt{-1}\left(f'(a_j^\la)(\gaa_j^\la)'(s^\la_*)\right)^2\neq0.
\eeq
For the general case of $\gaa_i^\lambda$ being of class $C^{k+1}$ near $a_j^\lambda$, we already know that $K^\la$ is of class $C^{k+1}$ with $k+1\geq2$. We apply Kellogg's Riemann mapping theorem to find a $C^{k+1}$ diffeomorphism  $R$ from the closure of the unit disc $\Delta$ onto $\overline\omega$, where $R$ is holomorphic in $\Delta$,  $\omega$ is an open subset of  $D^\lambda$ such that $\partial\omega$ is $C^{k+1}$. Moreover, $\partial\omega\cap\gamma_j^\lambda$ contains an open curve containing $z_0^\lambda$.  We choose $\omega$ so small that $f^\lambda\circ R$ is well-defined. Applying the Schwarz reflection principle to $f\circ R$ and repeating the above computation to $f\circ R$, we conclude that $f'(z_0^\la)\neq0$.
We now solve for $s=g(\la)$ from
\eq{xvj}
\DD{}{s}\tilde v(s,\la)=0
\eeq
for $(s,\la)$ close to $(s_*^{\la_0},\la_0)$. Note that $(s,\la)\mapsto \DD{}{s}\tilde v(s,\la)$ is of class $C^{k+\all,j}$.
By \re{non-zero-der}, $k\geq j$,   $k\geq1$, and the continuity of $\all_j^\la$ and the implicit function theorem, we conclude that the solution $g$ to \re{xvj} is of class $C^j$.
This shows that $\la\mapsto \all_j^\la$ is of class $C^j$. The proof for $C^\om$ case is by analogy.

$(iv).$ It follows directly from the above proof.
 \end{proof}

\vspace{2mm}
We point out that, as one will see from the counter-examples constructed in section~\ref{sec3}, additional properties are needed to prove some regularity results on families of biholomorphisms.
As an application of Theorem~\ref{1dmarked}, we obtain the following result giving global regularity of a family of biholomorphisms.

\begin{cor}\label{preg}Let  $j,k,\all, D,\cL D$ be as in \rpa{dpp}.
Let $\widetilde{\cL D}$ be another family of domains.  Suppose that  $\ov{\cL D},\ov{\widetilde{\cL D}}$ are in
$C^{k+1+\all,j}(\ov D)$   $($resp.~$C^{\om,\om}(\ov D))$.
Let $D$ be  $m$-connected with $m \geq 3$, where $\pd D\in C^{k+1+\all}$ $($resp.~$C^{\om})$.
Let $F:\cL D \rightarrow \widetilde{\cL D}$ be a family of biholomorphisms.
Suppose that one of following holds~$:$
\bppp
\item Assume that $F^\la(\gaa_i^\la)=\tilde\gaa_i^\la$ for $i=1,2$ and  there exists $p_\la\in \ov{D^\la}$ such that $\la\mapsto  (p_\la,F^\la(p_\la))$ is continuous.
  \item $Aut(D^\la)=\{\id\}$ for every $\lambda \in [0,1]$.
    \eppp
    Then $F\in  C^{k+1+\all,j}(\ov{\cL D})$
    $($resp. 
  $C^{\om,\om}(\ov {\cL D}))$.
\end{cor}
  \begin{proof} Let $\{\Gaa^\la\}$ and $\{\tilde \Gaa^\la\}$ be the parameterizations of $\cL D$ and $\tilde{\cL D}$ of the given regularity.

 $(i).$
 The problem is local in $\la$. Fix $\la_0$. If $p_{\la_0}\in\pd D^{\la_0}\in \gaa_\ell^{\la_0}$ for $\ell\neq1,2$. We have $F^\la(\gaa_\ell)=\tilde\gaa_{\ell'}$ where $\gaa_\ell,\tilde\gaa_{\ell'}$ vary continuous as sets. Swapping $\gaa_\ell,\tilde\gaa_{\ell'}$ with $(\gaa_2,\tilde\gaa_2)$, we may assume that $p_\la\in D^\la\cup\gaa_1^\la\cup\gaa_2^\la$.

By \rt{1dmarked}, we can find canonical conformal mappings
\begin{gather*}
K^{\la}\colon D^{\la}\to A^*_{m,r^\la}(\all^\la,\beta^\la),
\quad
\tilde K^{\la}\colon \tilde D^{\la} \to  A^*_{m,\tilde r^\la}(\tilde\all^\la,\tilde\beta^\la)
\end{gather*}
such that
$K^{\la}(\gaa_1^\la)$ is the unit circle, $|K^{\la}(\gaa_2^\la)| = r^{\la}_2$. Also, $\tilde{K}^{\la}$ is the unit circle and $|K^{\la}(\tilde\gaa_2^\la)| = \tilde{r}^{\la}_2$, for every $\la \in [0,1]$. By assumption, we have $\tilde r^\la_2=r^\la_2$.
Since $\tilde K^{\la} \circ F^{\la}$ is also a canonical mapping, then $\tilde K^{\la} \circ F^{\la}=\mu^{\la}K^{\la}$
 by \rt{1dmarked} $(ii)$, with $|\mu_{\la}| = 1$.
 Since $\la \mapsto(p^\la_1, F^{\la} (p^1_\la))$ is $C^0$, we know that the map $\la \mapsto \mu^{\la}$ is $C^0$.
 This shows that
 $F^\la=(\tilde K^\la)\circ(\mu^\la K^\la)$ is of class $C^{k+1+\all,0}$  in $(D\cup\gaa_1\cup\gaa_2)$.  For higher order derivatives in $\la$, let us verify the following

 \medskip
 \noindent
 {{\bf Claim.}   For each $\ell=3,\dots, m$, there exists $x_\ell^\la\in\gaa_\ell^\la$ so that $\la\mapsto (x_\ell^\la,F^\la(x_\ell^\la))$ is of class $C^j$.}

 \medskip
 To prove the Claim, we take a family of embeddings $\Gaa^\la$ from $\ov D$ onto
 $\ov{D^\la}$ with $\Gaa\in C^{k+1+\all,j}(\ov D)$.
 Take a subdomain $\Om$ in $D$ such that $D\setminus\ov\Om$ has $m$ components, while each component has boundary containing exactly one component of $\pd D$. Let $\Om^\la=\Gaa^\la(\Om)$ and $\widetilde\Om^\la=F^\la(\Om^\la)$. For each $\la$, the boundary of
 each component of
$ \tilde D^\la
 \setminus \ov{\widetilde\Om^\la}$
 has two components $C_i^\la,\tilde C_i^\la$, where $\tilde C_i^\la$ is one of $\tilde\gaa_3^\la,\dots, \tilde\gaa_m^\la$.
 We choose $C_i^\la$ so that it depends on $\la$ continuously, which is possible since $F^\la\circ\Gaa^\la$ is continuous in $D\times[0,1]$. The continuity implies that for each $i$, the $\tilde C_i^\la$ depends continuously on $\la$. Therefore, we have $F^\la(\gaa_\ell^\la)=\tilde\gaa_{i_\ell}^\la$ and $i_\ell$ does not depend on $\la$.

 Let $a_\ell^\la,b_\ell^\la$ be as in \rt{1dmarked} so that
 $$
 K^\la(\gaa_\ell^\la)=Arc(r_\ell^\la,[\all_\ell^\la,\beta_\ell^\la]).
 $$
 We also have  $
 \tilde K^\la(\gaa_\ell^\la)=Arc(\tilde r_\ell^\la,[\tilde\all_\ell^\la,\tilde\beta_\ell^\la]).
 $
Since $\tilde K^\la\circ F^\la=\mu^\la K^\la$, then $F^\la(\gaa_\ell^\la)=\tilde\gaa_{i_\ell}^\la$ implies that
$$
\mu^\la a_\ell^\la=\tilde a_{i_\ell}^\la.
$$
Since $\la\mapsto (a^\la_\ell,\tilde a_{i_\ell}^\la)$ is  of class $C^j$, we conclude that  $\la\mapsto \mu^\la$ is of class $C^j$ and
the Claim is valid for $x_\ell^\la=a^\la_\ell$.

 Using $F^{\la}=(\tilde K^{\la})^{-1}
\mu^{\la} K^{\la}$ again, we obtain that $\{F^{\la}\circ\Gaa^\la\}$ is in
$C^{k+1+\all,j}(D \cup \gamma_1 \cup \gamma_2)$.
Then from the Claim, it follows that  $\{F^{\la}\circ\Gaa^\la\}$ is in
$C^{k+1+\all,j}(D \cup \gamma_i \cup \gamma_\ell)$ for any distinct $i,\ell$.
Therefore, we have proved that $F\in C^{k+1+\all,j}(\ov{\cL D})$. The proof for the $C^\om$ case is by analogy.  The proof of Part $(i)$ is complete.
\end{proof}

The proof of Part $(ii)$ relies on the following Lemma.
The  result gives a sufficient condition for the convergence of a sequence of biholomorphisms between sequences of domains that converge, in the Hausdorff distance on sets.

\le{hurwitz} Let $D$ and $\tilde D$ be $m$-connected bounded domains in $\rr^2$, with $m \geq 3$. For every $j \geq 1$, let $D^j$ and $\tilde{D}^j$ be $m$-connected domains. We denote by $\gamma_1,\dots,\gamma_m$ $($resp.~$\tilde \gamma_1,\dots,\tilde \gamma_m)$ the components of $\pd D$ $($resp.~$\pd \tilde D)$ and by $\gamma^j_1,\dots,\gamma^j_m$ $($resp.~$\tilde \gamma^j_1,\dots,\tilde \gamma^j_m)$ the components of $\pd D^j$ $($resp.~$\pd \tilde{D}^j)$, for $j \geq 1$. We assume that

\bppp
\item   $\partial D \in C^{1+\alpha}$ and $\partial \tilde D \in C^{1+\alpha}$.
\item For every $j \geq 1$, $\partial D^j \in C^{1+\alpha}$ and $\partial \tilde{D^j} \in C^{1+\alpha}$.
\item $\lim_{j \rightarrow \infty}\overline{D^j} = \overline{D}$ and $\lim_{j \rightarrow \infty}\overline{\tilde{D}^j} = \overline{\tilde{D}}$ for the Hausdorff convergence of sets.
\item $Aut(D) = \{\id\}$.
\item For every $j \geq 1$, there is a biholomorphism $\phi_j : D^j \to \tilde{D}^j$.
\eppp
Then $\phi_j$ converges  to the unique biholomorphism $\phi$ from $D$ onto $\tilde{D}$  as $j\to\infty$,
uniformly on compact subsets of $D$.
Moreover, there exists a permutation $\nu$
of $\{1,\dots, m\}$ such that for every $i=1,\dots,m$, we have $\phi(\gaa_i) = \tilde\gaa_{\nu(i)}$ and for sufficiently large $j$, $\phi_{j}(\gaa_i^j) = \tilde\gaa_{\nu(i)}^{j}$.
 \end{lemma}

 We recall that given two non empty subsets $X,Y$ in a metric space $(E,d)$, the Hausdorff distance between $X$ and $Y$ is defined by
 $$
 d_H(X,Y):=\max\left(\sup_{x \in X}d(x,Y), \sup_{y \in Y}d(y,X)\right) \in \mathbb [0,\infty].
 $$
If $X$ and $Y$ are bounded domains in $\cc^n$ endowed with the Euclidean distance, then $d_H(X,Y) < \infty$.

We say that $\overline{D}^j$ converges to $\overline{D}$ for the Hausdorff convergence of sets if $$\lim_{j \to \infty}d_H(\overline{D^j},\overline{D}) = 0.$$
 Hence, if $\lim_{j \to \infty}d_H(\overline{D^Ï},\overline{D}) = 0$ then for every $j$ we may label the components $\gamma^j_1,\dots,\gamma^j_m$ of $\partial D^j$ such that for every $k=1,\dots,m$ we have:
 $$
 \lim_{j \to \infty}d_H(\gamma^j_k,\gamma_k) = 0.
 $$

  \begin{proof}
The uniqueness of $\phi$ is a direct consequence of $(iv)$.

Since all the domains are bounded, the sequence $\{\phi_j\}$ is a normal family. It suffices to show that, taking a subsequence if necessary,
the limit $\phi$ of $\phi_j$ is a biholomorphic mapping from $D$ onto $\tilde D$.
By Hurwitz's theorem, we know that
 $\phi$ is either a constant $b\in\pd \tilde D$ or it is a biholomorphic mapping from $D$
onto $\tilde D$.  See Courant~\cite{Co50}*{p. 191} for two possibilities in a different setting. In our case, we want to rule out the case that $\phi$ is a constant.

Suppose that $\phi=b$ is constant. Let $E$ be a closed simple curve contained in $D$ and close to $\gaa_1$ in the Hausdorff distance. Then $\gamma_1$ is contained in the bounded component of $\cc \backslash E$. Relabeling the components of $D^j$ if necessary, we may assume that the sequence $\gamma^j_1$ converges to $\gamma_1$; hence the bounded component of $\cc \backslash E$ contains
$\gaa_1^{j}$ for $j \geq 1$; for the latter, we may assume that $\gaa_1^j$ is not the outer boundary after applying a Mobius transformation to the family of domains. Using condition $(ii)$ and Kellogg's Theorem on the regularity of $\phi_j$, we know that $\phi_j(\gaa_1^j) = \tilde{\gaa}_1^j$ for every $j \geq 1$, relabeling the components of $\tilde{D}^j$ if necessary.
Moreover $\phi_j(E)$ bounds a bounded domain that contains $\tilde\gaa_1^{j}$ for every $j \geq 1$, by applying a Mobius transformation if necessary.
 Finally, if $a$ is a fixed point in the component of $\mathbb C \backslash\ov{ \tilde{D}}$ bounded by $\tilde{\gamma}_1$, then $a$ is in the component of $\mathbb C \backslash\ov{ \tilde{D^j}}$ bounded by $\tilde{\gamma}^j_1$, for sufficiently large $j$, according to $(iii)$.
Then, for every $j \geq 1$:
$$
\f{1}{2\pi i}\int_E\f{\phi_j'(\zeta)}{\phi_j(\zeta)-a}\, d\zeta = \f{1}{2\pi i}\int_{\phi_j(E)}\f{d\zeta}{\zeta-a} = \f{1}{2\pi i}\int_{\tilde{\gaa}_1^j}\f{d\zeta}{\zeta-a} = \pm 1.
$$
However, the left-hand side tends to $0$ as $j\to\infty$, which is a contradiction. Therefore,
$\phi$ is a biholomorphic mapping from $D$ onto $\tilde D$.  By Kellogg's theorem,
$\phi$ extends to a $C^{1+\all}$ diffeomorphism up to the boundary.

We may assume that
$
\phi(\gaa_i)=\tilde\gaa_i$ for every $j=1,\dots,m$.
We want to show that $\phi^{j}(\gaa_i^{j})=\tilde\gaa_i^{j}$ for   $i=1,\dots,m$ and all sufficiently large $j$.
Let us fix $1 \leq i \leq m$ and let us deform $\gaa_i$ slightly inside $D$ into another simple closed curve $E_i$ contained in $D$, so that $\gamma_i \cup E_i$ bounds a $2$-connected domain and $E_i$ separates $\gamma_i$ from $\gamma_l$, for $l \neq i$. It follows from $(iii)$ that $\gamma_i^j$ and $\gamma_i$ are on the same side of $E_i$ and $E_i$ separates $\gamma^j_i$ from $\gamma^j_l$ for $l \neq i$, for sufficiently large $j$. Since $\phi^{j}$ converges to $\phi$ in a neighborhood of $E_i$, then by considering orientation, we conclude that  $\phi^{j}(\gaa_i^{j})$ and $\phi(\gaa_i)$ are on the same side of  $\phi(E_i)$.
Therefore, $\phi^{j}(\gaa_i^{j})=\tilde\gaa_i^{j}$ for   $i=1,\dots,m$ and   sufficiently large $j$.
\end{proof}

Now Part $(ii)$ of Corollary~\ref{preg} follows from
  \rl{hurwitz} and Part $(i)$ of Corollary~\ref{preg}.

\setcounter{thm}{0}\setcounter{equation}{0}
\section{Negative results}\label{sec3}
In  the one-dimensional  case, we know  from Proposition~\ref{annulus} that counterexamples can be constructed only
for $m$-connected domains with $m\geq3$.

We say that two points $a,b$ in the extended complex plane $\hcc:=\cc\cup\{\infty\}$ are {\it symmetric}   with respect to a circle
if there is a Mobius transformation that preserves the circle and sends $a$ to the center of the circle
and $b$ to $\infty$.
 If $a$ is neither the center nor $\infty$, then $a$ is symmetric to $b$ w.r.t. the circle if and only if $a\ov b=r^2$, where $r$ is the radius of the circle.
Therefore if  two circles are concentric,   the center of the circles and $\infty$   form the symmetric pair   with respect to both circles and the symmetric pair is separated by the two circles. In fact the last two assertions hold for any two disjoint circles, since two disjoint circles are conformally equivalent to two concentric circles.
For an $m$-connected circular domain $D$, there are exactly $s_m:=m(m-1)/2$ pairs $\{a_j,b_j\}$ symmetric w.r.t a pair of circles on the boundary, while $a_j,b_j$ are in separate components of   $\hcc\setminus\ov D$.
Furthermore, $\{a_j,b_j\colon j= 1,\dots, s_m\}$ consists of $2s_m$ points.

\vspace{2mm}
Indeed, if $\{a_1,b_1\}$ is a symmetric pair w.r.t. two circles $S_1, S_2$, we may assume that $a_1=0,b_1=\infty$, while $D$ is contained
in the annulus bounded by $S_1, S_2$. Suppose that $S_1$ is the disk of the smaller radius  and assume that  $\{a_1,b_2\}$ is another symmetric pair w.r.t. two circles $S_3,S_4$ which are  components of $\pd D$. Since $a_1,b_2$ are in the complement of the annulus bounded by $S_3,S_4$, then $S_1$ must be one of $S_3,S_4$. Then $\{a_1,b_2\}=\{0,\infty\}$ and so $S_3,S_4, S_1$ are concentric; this is impossible.

Let $S_r(c):=\partial \mathbb D_r(c)$ be the circle in $\cc$ centered at $c$ with radius $r$.
When $D$ is an $m$-connected circular domain, we write $$\pd D=\cup_{j=1}^m S_{r_j}(c_j).
$$
We say that a sequence of  bounded circular $m$-connected
 domains $D_j$ converges to a bounded $m$--connected domain
  $D'$ if $\pd D_j=\cup_{l=1}^m S_{r_{j\ell}}(c_{j\ell})$ and $\pd D'=\cup_{l=1}^m S_{r'_{\ell}}(c'_{\ell})$ satisfy that $c_{j\ell}\to c_\ell'$ on $\cc$ and $r_{j\ell}\to r_\ell$ as $j\to\infty$.

We first need the following theorem of Koebe~\cite{Gr78}*{Theorem 3.9.1, p. 113}:
 Let $D$ be a  bounded  $m$-connected
domain with $\pd D\in C^{1+\all}$ in $\cc$. Fix $z_0\in D$. Then there exists a unique biholomorphic mapping $K$ from $D$ onto
$\hat D$ such that $K(z_0)=0$, $K'(z_0)=1$, and $\hat D$ is a circular planar domain (i.e. $\pd \hat D$
is a union of circles).

The uniqueness assertion implies that every biholomorphic mapping between
two circular domains is a Mobius transformation.  Furthermore, for each $m$-connected circular domain $D$ with $m\geq2$, an automorphism $\tau$ of $D$ is uniquely determined by $\tau(a)$ and $\tau(\gamma)$, where $\gaa$ is a connected component of $\pd D$ and $a$ is  $\pd D\setminus\gamma$.
When $\pd D$ consists of $m$ components $\gaa_1,\dots,\gaa_m$, with $m \geq 3$, for each permutation $\nu$ of $\{1,\dots, m\}$ there exists at most one
biholomorphism $\phi$ of $D$ such that $\phi(\gaa_i)=\gaa_{\nu(i)}$ for $i=1,\dots,m$.

\subsection{Automorphisms of multi-connected circular domains}
For distinct points $a,b,c$ in the Riemann sphere $\hat \cc$, let $\var_{a,b,c}$ be the Mobius transformation sending $0,1,\infty$ to $a,b,c$ respectively.
  We first have the following elementary result.
\le{mscon} Let $a,b,c$ be distinct points in   $\hat{\cc}$. Let $a_j$, $b_j$ and $c_j$ be sequences of complex numbers such that $a_j \to a$, $b_j \to b$ and $c_j \to c$, as $j \to \infty$, in spherical distance. Then $\var_{a_j,b_j,c_j}$ and  $\var_{a_j,b_j,c_j}^{-1}$ respectively converge to $\var_{a,b,c}$ and $\var_{a,b,c}^{-1}$ in the spherical distance, uniformly on   $\hat{\cc}$, as $j \to \infty$.
\ele

\begin{proof}
Fix distinct $a_*$, $b_*$, $c_*$ in  $\hat{\cc}$. To see the continuity, we need an expression of $\var:=\var_{a,b,c}$ or $1/\var$ when $(a,b,c)$ is close to a point $(a_*,b_*,c_*)$.  Set $\displaystyle q:=\f{b- c}{b- a}$.

\noindent{\bf Case 1.} $a_* \in \cc$.

\vspace{2mm}
  If $c_* \in \cc$, the Mobius transformation $\varphi:=\varphi_{a,b,c}$ is given by
$$
\var(z)=\f{c z -aq}{z-q}. 
$$
Then  $\var$ and $1/\var$ are continuous in all variables $a,b,c,z$, respectively when $z$ is away from $q$ and $z$ is close to $q$.
  This shows that $\var\colon\hcc\to\hcc$ is continuous in all variables, when $a_*,c_*$ are in $\cc$.

  If $c_*=\infty$, we rewrite $\displaystyle \frac{q}{c}=\frac{1-\f{b}{c}}{a-b}$. Hence we can write
$$
\var(z)=\f{z-a\f{q}{c}}{\f{1}{c}z-\f{q}{c}}.
$$
In particular, the maps $\var$ and $1/\var$ are continuous in all variables respectively
when $z$ is close to $a\f{q}{c}$ and it is away from $a\f{q}{c}$.

\vspace{2mm}
\noindent{\bf Case 2.} {\sl $a_*=\infty$.
}

\vspace{2mm}
We express
$$aq=\displaystyle\f{b-c}{\f{b}{a}-1},\quad
\f{1}{\var(z)}=\f{z-q}{cz-aq}.
$$
It is easy to see that $\var$ is continuous in all variables when $z$ is away from $q$, while $1/\var$ is continuous  in all variables for $z$ near $q$.

Finally, for the inverse mapping,  we have
 $$\var_{a,b,c}^{-1}(w)=\f{qw -aq}{w-c}.$$
 It is clear that $(a,b,c)\mapsto q$ is continuous. Then the continuity of $\var_{a,b,c}^{-1}(w)$ in $a,b,c,w$ follows from the continuity of $\var$ proved above.
\end{proof}

We have the following elementary result.
\le{semicon}
For every $j \geq 1$, let $D_j$ be a bounded $m$--connected circular domain, with $m \geq 2$, and let $\var_j$ be an automorphism  of $D_j$.
 Suppose that $D_j$ converge to an $m$-connected circular domain $D$. Then a subsequence of $\var_j$ converges uniformly on $\hcc$, in spherical distance, to an  automorphism of $D$.
\ele

  \begin{proof}
We first treat the case $m=2$ by a direct computation. Using Mobius transformations, we may find, for every $j \geq 1$, a real number $0 < r_j < 1$  and a biholomorphism $\phi_j$ from $D_j$ onto an annulus $A_{r_j}$. Furthermore, there is a Mobius transformation $\phi$ such that $\phi_j \to \phi$ uniformly on $\hcc$, as $j \to \infty$.
Then $\psi_j:=\phi_j \circ \var_j \circ \phi_j^{-1}$ is in the group generated by $I_{r_j}\colon z\mapsto r_j/z$ and $M_\mu\colon z\mapsto \mu z$ with $|\mu|=1$. Since the convergence of $D_j$ implies the convergence of $r_j$ to some $r$ satisfying $0 < r < 1$, it is easy to show that there is a subsequence of $\psi_j$ that converges to
an automorphism of $A_{r}$. Hence we may extract from $\var_j$ a converging subsequence.

Assume now that $m>2$.
Let $\{a^j_1,b^j_1\},\dots,\{a^j_{s(m)},b^j_{s(m)}\}$ be $s_{m}=m(m-1)/2$  pairs of symmetric points on $D_j$.  We know that the set $\{a^j_1,b^j_1,\dots,a^j_{s_m},b^j_{s_m}\colon j=1,\dots, s_m\}$ has exactly $2s_m$ points and $2s_m>2$.  Note that $\var_j$ preserves the set
$\{a^j_1,b^j_1,\dots,a^j_{s_m},b^j_{s_m}\colon j=1,\dots, s_m\}$.
By \rl{mscon}, we get a subsequence of $\var_j$ that converges in the spherical distance to an automorphism of $D$, uniformly on $\hcc$.
\end{proof}

Let $\mu, a, r$ satisfy
\eq{muar}
\quad 0<\mu<a-r<a+r<1.
\eeq

Then we may consider the 3-connected circular domain
\eq{T}
T(\mu,a,r)= \D\setminus \left\{\ov{\D_\mu(0)}\cup\ov{\D_r(a)}\right\}.
\eeq

Moreover, for $0\leq\e<1-a-r$, the domain $T(\mu,a,r+\e)$ is  still  circular.

Recall that a domain in $\cc^n$ is {\it rigid} if the identity map is the only automorphism   of the domain.

\le{rigid} Let  $\tau$ be the involution $z \mapsto \mu/z$ and let $T(\mu,a,r)$ be given by \rea{muar}-\rea{T}. \bppp
\item $Aut(T(\mu, a,r))=\{I,\tau\}$
 if and only if
\ga\label{muforI}
\mu=a^2-r^2, \\
  (a^2-r^2)^2-(r^{-1}-1)(a^2-r^2)+1\neq 0.
  \label{muforI+}
\end{gather}
\item $Aut(T(\mu, a,r))$ has the maximum 6 elements and  is generated by two non commutative involutions $\tau,\var_b$
where $\displaystyle b=\f{(a+r)-\mu}{1-(a+r)\mu}$ and $\displaystyle \var_b(z)= -\f{z-b}{1-bz}$, if {\rm(\ref{muforI})} holds and {\rm(\ref{muforI+})} does not hold.

\item Suppose that $\mu,a,r$ satisfy \rea{muforI}-\rea{muforI+}. Then $T(\mu,a,r+\e)$ is rigid when $\e$ is a sufficiently small  positive  real number.
\eppp
\ele

\begin{proof}
The union of $3$   symmetric pairs of $T(\mu,a,r)$ consists of
  $6$ distinct points on the real axis, which are in the complement of the closure of $T(\mu, a,r)$. Thus any Mobius transformation that  preserves the set of the six points must preserve the real axis too, of which $0,\infty$ are among the six points.

We first prove by a direct computation that $\tau \in Aut(T(\mu,a,r))$ if and only if (\ref{muforI}) holds. Indeed, if $\tau(S_r(a)) = S_r(a)$ then $\tau(a-r) = a+r$ which implies (\ref{muforI}). Now, if (\ref{muforI}) holds, then $\tau(a-r) = a+r$, $\tau(a+r) = a-r$ and $\tau(S_r(a))$ is a circle whose boundary contains the two points $a-r$ and $a+r$ and is orthogonal to the real axis at these two points. Then $\tau(S_r(a)) = S_r(a)$.

We can now prove $(i)$ and $(ii)$. Let $\var$ be an automorphism of $T(\mu,a,r)$, $\var \neq \tau$. There are different cases to  be considered.

\medskip
\noindent {\bf Case 1.} $\var(S_1(0)) = S_1(0)$.
\medskip

If $\var(1) = 1$ and $\var(-1) = -1$ then due to orientation we have necessarily $\var(-\mu) = -\mu$ since $\var$ preserves the real axis. Hence $\var = \id$.

If $\var(1) = -1$ and $\var(-1) = 1$ then $\var$ must interchange two other circles of the boundary for $\var$ not being the identity. Thus $\displaystyle \var(z)=c \f{z-b}{1-bz}$ where $b$ and $c$ are real and it follows that $c=-1$.
 Then we must have
$$
\var(z)= -\f{z-b}{1-bz},\quad b\in(-1,1).
$$
Since $\var^2=\id$, we have, due to orientation:
$$
\var(-\mu)=a+r, \quad\var(\mu)=a-r.
$$
By a simple computation,  we get
$$
b=\f{(a+r)-\mu}{1-(a+r)\mu}, \quad b=\frac{(a-r)+\mu}{1+\mu(a-r)}.
$$
This shows that
\eq{b=b}
\f{(a+r)-\mu}{1-(a+r)\mu}=\frac{(a-r)+\mu}{1+\mu(a-r)}.
\eeq
Since $\mu=a^2-r^2$, the above equality is equivalent to the fact that \re{muforI+} does not hold.

\medskip
\noindent {\bf Case 2.} $\var(S_1(0)) = S_{\mu}(0)$. Then $\tau \circ \var$ is an automorphism of $T(\mu,a,r)$ that preserves $S_1(0)$. Hence from Case 1 we conclude that either $\var = \tau$ or $\var = \tau \circ \var_b$.
 Note that the latter occurs only if \re{b=b} holds.

\medskip

\noindent {\bf Case 3.} $\var(S_1(0)) = S_r(a)$.
We then have two cases:  $(a)$ $\var(S_r(a))=S_\mu(0)$ and hence $\var(S_\mu(0))=S_1(0)$, or (b) $\var(S_\mu(0))=S_\mu(0)$. When $(a)$ occurs, we have by Case 2 $\var^{-1}=\tau$ or $\var^{-1}=\tau\circ\var_b$.  (The former actually does not occur.) When  $(b)$ occurs, we have by Case 2 $\var\tau=\tau$, which is impossible or $\var\tau=\tau\var_b$.  Chasing after the images of boundary components of $T(a,\mu,r)$ under the compositions of $\tau,\var_b$, we can verify that the automorphism group has six elements: the identity mapping, $ \tau,\var_b,\tau\var_b,\var_b\tau$ and finally $\tau\var_b\tau=\var_b\tau\var_b$.

\medskip

We can now prove $(iii)$. Suppose that there is $\e_j\to0$ with $\e_j\neq0$ such that $T(\mu, a,r+\e_j)$ admits a non-trivial automorphism $\var_j$. We first note that $\var_j$ cannot preserve each component of $T(\mu, a,r+\e_j)$. Otherwise $\var_j$ preserves the annulus bounded by $S_\mu(0)$ and $S_1(0)$. So it is a rotation and the only
rotation that preserves $T(\mu, a,r+\e_j)$ is the identity. By Lemma~\ref{semicon}, we may assume that $\var_j$ converges to an automorphism $\var$ of $T(\mu,a,r)$, uniformly on $\hcc$.
Thus $\var$ is not the identity and so $\var=\tau$ by $(i)$. Then $\var_j$ must interchange $S_\mu(0)$ and $S_1(0)$. Thus $\var_j$ is the composition of $\tau$ with a rotation. Clearly $\var_j$ cannot preserve $S_r(a+\e)$ for $\e \neq 0$.
\end{proof}

\subsection{Proof of Theorem~\ref{negative} for $n=1$}

Let us define two families of smooth domains $(D^{\la})_{-1 \leq \la \leq 1}$ and $(\tilde{D}^{\la})_{-1 \leq \la \leq 1}$ as follows:
 \begin{gather*}
 D^{\la} = T\left(\mu,a,r+ e^{-\f{4}{\la^2}}\right) \ {\rm for}\ -1 \leq \la \leq 1,\\
  \tilde D^{\la} = D^{\la} \ {\rm for} \ 0\leq \la \leq 1\ \ {\rm and} \ \ \tilde D^{\la} = \tau D^{\la} \ {\rm for} \ -1\leq \la \leq 0,
 \end{gather*}
for $\mu=\f{3}{16}$, $a=\f{1}{2}$ and $r=\f{1}{4}$.
  Note that $\mu,a,r$ satisfy \re{muforI}-\re{muforI+}.

\vspace{2mm}
We first point out that $\psi_{\la}=\id$ for $\la \geq 0$ and $\psi_{\la}=\tau$ for $\la < 0$ are biholomorphisms from $D^{\la}$ to $\tilde{D}^{\la}$.
We have the following precise version of Theorem~\ref{negative} for $n=1$.

  \begin{proof}
We first prove that the two families $\cL D$ and $\widetilde{\cL D}$ are smooth families.
We consider the following parameterization of $\partial D^{\la}$, denoted by $\Gamma^{\la} :  \partial{T}\left(\mu,a,r\right) \to \partial{T}\left(\mu,a,r+ e^{-\f{4}{\la^2}}\right)$ with
$$
\left\{
\begin{array}{cll}
\Gamma^{\la} & = & \id \ {\rm on}\ S_1(0),\\
\Gamma^{\la} & = & \id  \ {\rm on}\ S_{\mu}(0),\\
\Gamma^{\la}(a + r e^{it}) & = &  a + \left(r+e^{\f{-4}{{\la}^2}}\right)e^{it},\ \forall\,  t \in [0,2\pi].
\end{array}
\right.
$$
Then $(x,\la)\mapsto\Gamma^{\la}(x)$ is smooth of class $C^{\infty}$ on $(S_1(0) \cup S_{\mu}(0) \cup S_r(a)) \times [-1,1]$. Hence, according to~\cite{BG14}, $\{\Gamma^{\la}\}$ admits a smooth extension,  still denoted $\Gamma^{\la}$, as embeddings $\Gaa^\la$ from $\overline{T\left(\mu,a,r\right)}$ to $\overline{T\left(\mu,a,r+ e^{-\f{4}{\la^2}}\right)}$.

We consider now the following parameterization 
$\tilde \Gamma^{\la} : \partial{T}\left(\mu,a,r\right) \to \partial \tilde{D}^{\la}$ with
\gan
\left\{
\begin{array}{cll}
\tilde \Gamma^{\la} & = & \id\ {\rm on}\ S_1(0),\\
\tilde \Gamma^{\la} & = & \id\ {\rm on}\ S_{\mu}(0),\\
\tilde \Gamma^{\la}(a + r e^{it}) & = &  a + \left(r+e^{\f{-4}{{\la}^2}}\right)e^{it},\quad \forall\,  t \in [0,2\pi],\  \forall\,  \la \geq 0,\vspace{.75ex}
\\
\tilde \Gamma^{\la}(a + r e^{it}) & = &  \frac{\mu}{a + \left(r+e^{\f{-4}{{\la}^2}}\right)e^{it'}},\quad\quad\ \ \, \forall\,  t \in [0,2\pi],\ \forall\,  \la < 0,
\end{array}
\right.
\end{gather*}
where, for every $t \in [0,2\pi]$, the $t'$ is the unique solution on $[0,2\pi]$ of the equation
\begin{equation}\label{real}
a + re^{it} = \frac{\mu}{a + re^{it'}}.
\end{equation}
In particular, for every $t \in [0,2\pi]$:
$$
e^{it'} = -e^{it} \frac{a+re^{-it}}{a+re^{it}}
$$
and the parametrization $\displaystyle (t,\la) \mapsto \frac{\mu}{a + \left(r+e^{\f{-4}{{\la}^2}}\right)e^{it'}}$ is smooth for $\la < 0$.

Hence, to prove that the mapping $(x,\la)\mapsto\tilde{\Gamma}^{\la}(x)$ is smooth of class $C^{\infty}$ on $(S_1(0) \cup S_{\mu}(0) \cup S_r(a)) \times [-1,1]$, we need to verify that
$$
a + \left(r+e^{\f{-4}{{\la}^2}}\right)e^{it} = \frac{\mu}{a + \left(r+e^{\f{-4}{{\la}^2}}\right)e^{it'}}
$$
up to infinite order (with respect to $\la$) at $\la = 0$.
However
$$
a + \left(r+e^{\f{-4}{{\la}^2}}\right)e^{it} = a + re^{it} + O(\la^{\infty}),
$$
  i.e. the function $\la \mapsto a + \left(r+e^{\f{-4}{{\la}^2}}\right)e^{it} - (a + re^{it})$ vanishes up to infinite order at the origin, and
$$
\frac{\mu}{a + \left(r+e^{\f{-4}{{\la}^2}}\right)e^{it'}} = \f{\mu}{a + re^{it'}} + O(\la^{\infty}).
$$
The regularity of $\tilde{\Gamma}^{\la}$ on $S_r(a)$ follows now from (\ref{real}).

Then $\{\tilde \Gamma^{\la}\}$ admits a smooth extension, still denoted by $\{\tilde \Gamma^{\la}\}$, as an embedding from $\overline{T(\mu,a,r)}$ to $\overline{\tilde{D}^{\la}}$, according to \cite{BG14}.  Hence the two families of domains are smooth.

\vspace{2mm}
For a sufficiently small positive number $\la_0$ and $0<\la\leq\la_0$,
 the   automorphism group of $D^{\la}$ is $\{\id\}$ according to Lemma~\ref{rigid} $(iii)$; hence $\var^{\la}$ is smooth with respect to $\la$ for $\la > 0$. For $-\la_0\leq\la < 0$,
the   biholomorphism from $D^{\la}$ to $\tilde{D}^{\la}$ is $\tau$ according to Lemma~\ref{rigid} $(iii)$. Hence $\var^{\la}$ is smooth with respect to $\la$ for $\la < 0$. Finally, from the values of $\var^{\la}$, $\var^{\la}$ can be continuous at $\la=0$ only for $z$ such that $\tau(z) = z$, that is $z = - \sqrt{\mu}$ since $a-r < \sqrt{\mu} < a+r$, meaning $\sqrt{\mu} \in \D_{r}(a)$.

The above proof is complete for $3$-connected domains in \rt{negative}. Suppose now that $m\geq4$. Fix $\mu$ and $\tau$ as above. We consider   real numbers
$$
\mu<r_1<\dots<r_{2m-2}<1,
$$
with $r_{2m-2-i}=\tau(r_i)$. Let $C_i$ be the circle which is perpendicular to real axis and passes through $r_{2i-1},r_{2i}$.  Let $T_{\mu,r}$ be the domain bounded by $C_1,\dots, C_{m-1}$, $S_\mu(0)$ and $S_1(0)$.
We want to show that if $\e$ is positive and sufficiently small, the domain $T^\e:=T_{\mu,r^\e}$ is rigid, for
$$
r^\e=(r_1+\e,r_2+\e^2,r_3,\dots, r_{2m-2}).
$$
Let $\var$ be an automorphism of $T^\e$. We want to show that $\var=\id$ when $\e>0$ is sufficiently small.  Let $C_1^\e$ be the circle perpendicular to the real axis and passes through $r_1^\e,r_2^\e$. There are three cases.

\medskip

\noindent
{\bf
Case $1$.}
$\var(C_1^\e)=C_1^\e$. Then $\var$ must be a Mobius transformation preserving the set $E:=\{\pm1,\pm\mu,r_3,\dots, r_{2m-2}\}$. There are only finitely many such mappings $\var$  and we list all of them as $\var_0,\dots, \var_N$ with $\var_0$ being the identity. Note that the list does not require  $\var_i$   be an automorphism of $T^\e$.
We write
$$
\var_j(z)=\f{a_jz+b_j}{c_jz+d_j}.
$$
Then $\var$ is one of $\var_1,\dots, \var_N$, say $\var=\var_1$. Let us drop the subscript in $a_1,b_1,c_1,d_1$. Since $\var$ preserves the real axis, then $\var$ either interchanges $r_1-\e,r_2+\e$, or fixes both $r_1-\e$ and $r_2+\e$. For the latter, we want to show that $\var$ is the identity when $\e$ is positive and sufficiently small. Indeed, let $\psi$ be a  Mobius transformation  sending $(r_1+\e,r_2+\e^2)$ to $(0,\infty)$ and $r_1+\e$ to $0$.
Then $\psi\var\psi^{-1}(z)=\la z$. Note that $E$ is outside $C_1^\e$. Hence $\psi(E)$ is contained in the negative real
axis. Since $\psi\var\psi^{-1}$ preserves $\psi(E)$, then it is the identity map. This contradicts that $\var$ is not the identity. Therefore, we are left with the case
 $\var(r_1-\e)=r_2+\e^2$; equivalently, we have
$$
\f{a(r_1+\e)+b}{c(r_1+\e)+d}=r_2+\e^2.
$$
The equation admits finitely many solutions in $\e$.

\medskip

\noindent
{\bf
Case $2$.}
$\var(\tilde C)=C_1^\e$ and $\tilde C\neq C_1^\e$. In this case we find components $C,C'$ of $\pd T^\e$ other than $C_1^\e$ so that $\var(C)=C'$. Suppose that the intersection of $C$ (resp.~$C',\tilde C$) with the real axis
is $\{\all,\beta\}$ (resp.~$\{\all',\beta'\}$, $\{\tilde\all,\tilde\beta\}$). In fact, we take
$\all'=\var(\all)$ and $\beta'=\var(\beta)$.  Define
$$
K_{\all,\beta}(z)=\f{z-\all}{z-\beta},\quad K_{\all,\beta}^{-1}(z)=\f{\beta z-\all}{z-1}.
$$
Then $K_{\all',\beta'}\var K_{\all,\beta}^{-1}$ must be a dilation:
$$
\var(z)=K^{-1}_{\all',\beta'}(\del K_{\all,\beta}(z)).
$$
Now $\var(\tilde C)=C_1^\e$ implies that $\var(\tilde\all)=r_1+\e$ and $\var(\tilde\beta)=r_2+\e^2$. Therefore,
\gan
\f{\beta'\del K_{\all,\beta}(\tilde\all)-\all'}{\delta K_{\all,\beta}(\tilde\all)-1}=r_1+\e,\\
\f{\beta'\del K_{\all,\beta}(\tilde\beta)-\all'}{\delta K_{\all,\beta}(\tilde\beta)-1}=r_2+\e^2.
\end{gather*}
Since $\tilde\all\neq\all,\beta$ then $K_{\all,\beta}(\tilde\all)\neq0,\infty$. Moreover, $\beta'\neq r_1$ because $C'\neq C_1^\e$ and $\e$ is small.
Thus we can solve for $\del$ from the first equation and rewrite the second equation as
$$
\f{a_*\e+b_*}{c_*\e+d_*}=r_2+\e^2,
$$
where $a_*,b_*,c_*,d_*$ are finite and independent of $\e$ and $c_*=K_{\all,\beta}(\tilde\all)\neq0$. The above equation has only finitely many solutions for $\e$.
Again, we can verify that $C_1^\la,C_2,\dots, C_{m-1}$, $S_1(0)$ and $S_\mu(0)$ bound a smooth family of domains $D^\e$ for $\e>0$. Furthermore, $D^\e$ are rigid when $\e$ is positive and small, while $D^\la$ admits
the non-trivial automorphism $\tau$. The rest of the proof is similar to the proof for  case $3$-connected domains.
\end{proof}

\subsection{Proof of Theorem~\ref{negative} for $n\geq 2$}

We use a construction of rigid domains by Burns-Shnider-Wells~\cite{BSW78}, perturbing  the unit ball.  Using  a Cayley transformation $\Psi$, we identity  $ \cL U_{n}=\{(z,w)\in \cc^{n-1} \times \cc \colon \IM w>|z|^2\}$ with the unit ball.  Let $f(z,u+iv)=u^4+|z|^4$. Fix a real-valued polynomial
$$
N(z,\ov z, u)=\sum_{m\leq \ell\leq m'}\sum_{|I|+|J|+2k=\ell}N_{IJ k}z^I\ov z^J u^k
$$
such that the only matrix $U\in U(n-1)$ satisfying  $N(Uz,\ov{Uz},u)=N(z,\ov z,u)$ is the identity matrix.  Furthermore,
\eq{vzN}
v=|z|^2+N(z,\ov z,u)
\eeq
has  pseudohermitian curvature
that does not vanish on any non-empty open set
with $u^4+|z|^4\leq1$. For instance, it suffices to choose an $N$  such that
\eq{trac42}
\operatorname{Tr}^2\sum_{|I|=4,|J|=2}\sum_{k=0}^\infty N_{IJk}z^I\ov z^Ju^k\not\equiv0,
\eeq
 where Tr denotes the trace  (see~\ci{CM74}*{p.~268} for detail.)
 For an  example of $N$, we can   choose   $\e_0>0$ sufficiently small and
$$
N(z,\ov z,u)=\e_0\sum_{j=1}^{n-1}2\RE\{z_j^4\ov z_j^2(1+\ov z_j)\}u^{j-1}.
$$
Then  \re{vzN} defines a strictly pseudoconvex real hypersurface.
 The trace given by \re{trac42} equals $\e_0\sum 2\RE(z_j^2)u^{j-1}$, which is not identically zero.
Define
$$
\var^{\lambda}(z,u)=\begin{cases}\e\exp^{ \dfrac{- f(z,u)}{{1-f(z,u)}}-\dfrac{1}{{|\lambda|}}}&\text{if \ $f(z,u)<1$ and $\lambda\neq0$},\\
 & \\
0&\text{otherwise},
\end{cases}
$$
with $\e>0$. For $\lambda \in \mathbb R$, we consider the domain $\Omega^{\lambda}$ defined by
\begin{equation*}
\displaystyle{
\Omega^{\lambda}=\{(z,w) \in \mathbb C^{n-1}\times \mathbb C : v>|z|^2+N(z,\ov z,u)\var^\la(z,u)\}.
}
\end{equation*}
For $\la = 0$, $\Omega^{\la} = \cL U_{n}$. When $\la \neq 0$ and $\e$ is sufficiently small,    the real analytic hypersurface $\pd \Omega^{\la}\cap\{f<1\}$ remains non-spherical
 as pseudohermitian curvature is not identically zero (see~\ci{CM74}*{p.~260}).

Using the Cayley transform $\Psi$, we pull back $\Omega^{\la}$ to a strictly convex domain $D^{\la} \subset \subset \cc^{n}$, as the pole of the Cayley transform does not intersect   $\pd\Om^\la$ when $\e$ is sufficiently small. For $\la = 0$, $D^{\la}$ is the unit ball in $\cc^n$. For $\la \neq 0$,  the real hypersurface $S_-^\la:=\Psi^{-1}(\pd \Omega^{\la}\cap\{f<1\})$
remains non-spherical and $S^{\lambda}_+:=\Psi^{-1}(\pd D^{\la}\setminus \Psi^{-1}(\{f>1\}))$ is part of the unit sphere.

We now apply an argument in~\cite{BSW78}*{sect.~6},  valid for  $\cc^n$ for each $n\geq2$ and the above example of $N$. For the reader's convenience, we recall their proof here. Fix $\la>0$. Let $T$ be an automorphism of $D^\la$. By Fefferman's theorem, it extends to a smooth CR automorphism of $\pd D^{\lambda}$. Since $S_-^{\lambda}$ is non-spherical in any open subset by \re{trac42}, then
$$
T(S_-^{\lambda})\cap S_+^{\lambda}=\emptyset.
$$
This shows that $TS_+^{\lambda}=S_+^{\lambda}$ and $T\pd S_+^{\lambda}=\pd S_+^{\lambda}$. It is proved in~\cite{BSW78} that $T(z,w)=(Uz,w)$   for every $(z,w) \in \cc^{n-1} \times \rr$, where $U$ is a unitary matrix. We now have
$$
N(Uz,\ov U\ov z,u)\var^{\la}(Uz,u)=N(z,\ov z,u)\var^{\la}(z,u),
$$
for $|z|^4+u^2<1$.  Note that $\var^\la(Uz,u)=\var^\la(z,u)$. Therefore, $N(Uz,\ov U\ov z,u)=N(z,\ov z,u)$ and we obtain $U=\id$.

Let $\tilde D^{\la}=D^{\la}$ for $\la\geq0$ and $\tilde D^\la= \tau D^{\la}$ for $\lambda<0$. Here $\tau: z \mapsto -z$.
Note that the smooth function
$\var^\la
$
 vanishes to infinite order at $\lambda=0$ or $f=1$.  Thus  $\{{D^{\la}}\}$, $\{\tilde D^{\la}\}$ are  smooth families of smoothly bounded strictly convex domains.

 \medskip

 The proof of \rt{negative} is complete.

\setcounter{thm}{0}\setcounter{equation}{0}

\section{Positive results in higher dimensions}\label{sec4}

One of the features of the counterexample in Theorem~\ref{negative}  is that  the fiber at $\la = 1/2$, $D^{1/2}$, has a non trivial automorphism group. The following theorem shows that this non rigidity property is the obstruction for the  continuous extension.

\th{boundary} Let $\cL D$, $\widetilde{\cL D}$ be two continuous families of  bounded strictly pseudoconvex domains of $C^{2}$ boundaries in $\cc^n$.  Let $F: \cL D \to \widetilde{\cL D}$ be a family of biholomorphisms. If $Aut(D^0) = \{\id\}$, then the map $F^\la\circ\Gaa^\la$ converges  to $F \circ\Gaa$ uniformly in $\ov D$ as $\la\to0$, for any parameterization $\Gaa\in C^{2,0}(\ov D)$ of $\ov {\cL D}$.
\eth

   The case $n=1$ is proved in \rl{hurwitz}.  We now assume $n\geq2$.
The proof consists of following steps. We first establish the convergence of $F^\la$ in compact sets as $\la\to0$ by using Pinchuk's scaling method.  As a consequence we control the distance to the boundary for $F^\la$. Combining it with the uniform Hopf lemma, we obtain a precise control of the distance from $F^\la(x)$ to $\pd\tilde D^\la$.
It is worthy pointing out that when domains are fixed, the scaling step is not necessary. This is one of the differences when we deal with a family of domains vs individual domains. In fact, by Henkin's theorem we already know that each biholomorphism   $F^{\la}$ is already   H\"older-$\f{1}{2}$ continuous  up to the boundary of  $D^{\la}$.
The uniform $C^{1/2}$ estimate is obtained by a uniform estimate on the Kobayashi metric.

We first point out that this is a local result, meaning that we only need to consider $\lambda$ close to the zero.
We now provide the proof via several lemmas.

\le{claim}
As $\la\to0$, the mapping $F^\la $ converges to $F^0$ uniformly in each compact subset of $D$.
\ele
\begin{proof}
 Note that it will follow from the Cauchy inequalities that every derivative of $F^{\la}$ will converge in compact sets, when $\la$ goes to 0, to the corresponding derivative of $F^0$.

 The method to prove the lemma is now classical and relies on the scaling method introduced by Pinchuk~\cite{Pi}. For the convenience of the reader we present this in detail. To prove the lemma, it suffices to prove that every sequence of biholomorphisms from $D^{\la}$ to $\tilde D^{\la}$ admits a subsequence that converges to $F^0$. Indeed, that will imply that the family $(F^{\la})$ converges, in the compact-open topology on $D^0$, to $F^0$.

Consider a sequence $(F^{\la_k})$ of maps of $F$, with $\la_k \to 0$ as $k \to \infty$.
Since the domains $\tilde D^{\la}$ are contained in a fixed bounded domain, we may extract from $(F^{\la_k})$ a subsequence that converges, uniformly on compact subsets of $D^0$, to some holomorphic map from $D^0$ to $\ov{\tilde D^0}$. Since maps $F^{\la_k}$ are biholomorphisms from $D^{\la_k}$ onto $\tilde D^{\la_k}$, then the subsequence either converges to $F^0$, the unique biholomorphism from $D^0$ to $\tilde D^0$, or
there is a point $z^0\in D^0$ such that $F^{\la_k}(z^0)$ approaches to the boundary of $\tilde D^0$.
Without loss of generality, we may assume for the latter that
\eq{z0Lim}
\lim_{k \to \infty}F^{\la_k}(z^0) = p.
\eeq

We now seek a contradiction to \re{z0Lim}.
By assumption, there exist $0 <\lambda_0 < 1$ and a domain $U_0 \subset \subset \mathbb C^n$ such that for every $0 \leq \lambda \leq \lambda_0$ we have the inclusion $\overline{D^{\lambda}} \subset U_0$. Moreover, we may consider for every $0 \leq \lambda \leq \lambda_0$ a defining function $r_{\lambda}$ of $D^{\lambda}$ such that:
\begin{itemize}
\item[(i)] $r_{\lambda}$ is strictly pluri-subharmonic on $\overline{U_0}$,
\item[(ii)] $\lim_{\lambda \rightarrow 0}\|r_{\lambda}-r_0\|_{ C^2(\overline{U_0})}= 0$.
\end{itemize}
Shrinking $\la_0$ if necessary, there is a domain $\tilde U_0 \subset \subset \mathbb C^n$ such that for every $0 \leq \lambda \leq \lambda_0$ we have the inclusion $\overline{\tilde D^{\lambda}} \subset \tilde U_0$. Moreover, we may consider for every $0 \leq \lambda \leq \lambda_0$ a defining function $\tilde r_{\lambda}$ of $\tilde D^{\lambda}$ such that:

\begin{itemize}
\item[(iii)] $\tilde r_{\lambda}$ is strictly pluri-subharmonic on $\overline{\tilde U_0}$,
\item[(iv)] $\lim_{\lambda \rightarrow 0}\|\tilde r_{\lambda}-\tilde r_0\|_{ C^2(\overline{\tilde U_0})}= 0$.
\end{itemize}

Next, we note that the domain $\tilde D^0$, being strictly pseudoconvex,  admits a (local) peak holomorphic function $\tilde \varphi$ at $p$.  Namely, there exists a neighborhood $\tilde V_0$ of $p$ in $\mathbb C^n$, contained in $\tilde U_0$, such that $\tilde \varphi$ is holomorphic on $\tilde D^0 \cap \tilde V_0$ and continuous on $\overline{\tilde D^0} \cap \tilde V_0$, and
$$
\left\{
\begin{array}{lll}
\tilde \varphi(p) & = & 1,\vspace{.75ex}\\
|\tilde \varphi(z)| & < & 1,\ \forall\,  z \in (\overline{\tilde D^0} \cap \tilde V_0)\backslash \{p\}.
\end{array}
\right.
$$
Hence, according to (\ref{z0Lim}), the sequence $(F^{\la_k})_k$ converges in the compact-open topology on $D^0$ to the point $p$,
by applying the maximum
principle to $\tilde\var\circ\lim_{k\to\infty}F^{\la_k}$ in a sufficiently small neighborhood of $z_0$ (see \cite{Pi}).

Shrinking $\tilde V_0$ if necessary, there exists a biholomorphism $\tilde \Phi$ from a neighborhood of $\tilde V_0$ to its image, with $0 \in \tilde V_1:=\tilde \Phi(\tilde V_0)$, such that $\tilde \Phi(p) = 0$ and
$$
 \tilde r_0 ( \tilde \Phi^{-1}(z),\overline{ \tilde \Phi^{-1}(z)}) = \RE (z_n) + |z'|^2 + \varphi_0(z,\ov z)
$$
on $\tilde V_1$, where $|\varphi_0(z,\ov z)|=o(|z_n| +  |z'|^2)$ on $\tilde V_1$.

For sufficiently large $k$, we have $z^0 \in D^{\la_k}$ and there is a unique point $p_k \in \partial
(\tilde\Phi
 \tilde D^{\la_k}) \cap \tilde V_1$ such that
$$
\tilde \delta_k:=\dist(\tilde \Phi \circ F^{\la_k}(z^0),  p_k) = \dist(\tilde \Phi \circ F^{\la_k}(z^0), \partial
(\tilde\Phi
\tilde D^{\la_k})),
$$
where $\dist(x, S)$ denotes the Euclidean distance from $x$ to subset $S$ of $\cc^n$.

Here we have used Condition $(iv)$ in the proof of \rl{claim}, namely that $\tilde r_{\la_k}$ converges to $\tilde r_0$ in $C^2$ norm on $\ov{\tilde V_0}$, and the fact that $\tilde \Phi \circ F^{\la_k}(z^0) \in \tilde V_1$ for sufficiently large $k$.

Let, for sufficiently large $k$, $\tilde \Psi_k$ denote the composition of a translation and of a unitary map such that $\tilde \Psi_k(p_k) = 0$ and $\tilde \Psi_k \circ \tilde \Phi \circ F^{\la_k}(z^0) = (0',-\tilde \delta_k)$.
Note that up to a unitary transform of $\mathbb C^n$, we may choose $\tilde \Psi_k$ so that it converges to identity in $C^2$ norm on $\ov{\tilde V_1}$. Hence, restricting $\tilde V_1$ if necessary, we may assume that for every sufficiently large $k$:
\begin{equation}\label{la-eq}
r_k^*(z,\ov z):= \tilde r_{\la_k}\left(\tilde \Psi_k^{-1}(z),\overline{\tilde \Psi_k^{-1}(z)}\right) = \RE(z_n) + |q_k(z')|^2 + \psi_k(z,\ov z),
\end{equation}
on $\tilde \Psi_k(\tilde V_1)$, where
$\psi_k(z,\ov z)=o(|z_n| +  |z'|^2)$
on $\tilde \Psi_k(\tilde V_1)$ and
  $q_k(z')$
  is a quadratic polynomial that
  converges to $|z'|^2$
  as $k\to\infty$.
Here, we still have used Condition (iv) in the proof of \rl{claim}.

Now, let $\tilde \Lambda_{k}$ be the dilation map
$$
\tilde\Lambda_k : z \in \cc^n \mapsto \left(\sqrt{\tilde \delta_k}z',\tilde \delta_k z_n\right).
$$
Then the map $\tilde \Lambda_k ^{-1}\circ \tilde \Psi_k \circ \tilde \Phi$ is a biholomorphism from $\tilde V_0$ onto $\tilde \Lambda_k ^{-1} \circ \tilde \Psi_k (\tilde V_1)$.

It is straightforward that $\lim_{k \to \infty}\tilde \Lambda_k^{-1} \circ \tilde \Psi_k (\tilde V_1) = \cc^n$. Moreover, by (\ref{la-eq}) and the convergence of $\tilde\delta_k^{-1}r_k^*(\tilde\Lambda_k(z),\ov{\tilde\Lambda_k(z)})$   to $x_n+|z'|^2$ as $k\to\infty$ on any compact subset of $\cc^n$, the sequence of domains $\tilde \Lambda_k ^{-1} \circ \tilde \Psi_k(\tilde D^{\la_k} \cap \tilde V_1)$ converges to $\mathbb H_n$, for the local Hausdorff convergence of sets in $\cc^n$, i.e. for every ball $K$ in $\cc^n$, centered at  0, the sequences of domains $(\tilde \Lambda_k ^{-1} \circ \tilde \Psi_k(\tilde D^{\la_k} \cap \tilde V_1)))\cap K$ converges to $\mathbb H_n\cap K$ in the Hausdorff distance. Here $\mathbb H_n:=\{(z',z_n) \in \cc^n \colon \RE (z_n) + |z'|^2 < 0\}$ is the unbounded representation of the unit ball.

Since $F^{\la_k}$ converges to $p$ uniformly on compact subsets of $D^0$ then,
w.r.t.  the Hausdorff convergence of sets in $\cc^n$, the sequence $(F^{\la_k})^{-1}(\tilde D^{\la_k} \cap \tilde V_0))$,
which is contained in $D^{\la_k}$,
converges to $D^0$.

Let, for sufficiently large $k$, the map $F_k$ be defined by
$$
 F_k:=
\tilde \Lambda_k
 ^{-1}
  \circ \tilde \Psi_k \circ \tilde \Phi \circ F^{\la_k}.
$$
Since $F_k(z^0) = (0',-1)$ and the sequence $ F_k
  (D^{\la_k} \cap \tilde V_0)$ converges to $\mathbb H_n$, it is standard that the sequence $F_k$ is a normal family and admits a subsequence, still denoted by $F_k$, that converges to some holomorphic map $F_0$ from $D^0$ to $\mathbb H_n$, with $F_0(z^0) = (0',-1)$. See for instance \ci{Pi}.

Moreover, since the domains $D^{\la_k}$ are all contained in a bounded domain, the sequence $
F^{-1}_k
$ admits a subsequence, still denoted by $F^{-1}_k$, that converges to some holomorphic map $G_0$ from $\mathbb H_n$ to $D^0$, with $G_0((0',-1)) = z^0$.

Finally, from the equalities $F_k\circ F_k^{ - 1} = I$ on $
F_k( D^{\la_k} \cap \tilde V_0)
$ and $F_k^{-1} \circ F_k = I$ on $
  D
  ^{\la_k}\cap \tilde V_0$,
and $F_k(z^0)=(0',-1)$,
 we obtain 
 that $G_0 = F_0^{-1}$, meaning that $D^0$ is biholomorphic to the unit ball in $\cc^n$. 
This is a contradiction and the lemma is verified.
\end{proof}

\vspace{2mm}
We remark that the lemma holds if $D^0$ is not biholomorphic to a ball.
An immediate consequence of \rl{claim} is the following: For every $\e>0$, there exists $\la_0 > 0$   so that for every $0 \leq \la \leq \la_0$:
\eq{uniform-dist}
\dist(F^\la(z),\pd \tilde D^\la)>\del, \quad \text{if $z\in D^\la$ and $\dist(z,\pd D^\la)>\e$}.
\eeq
Furthermore, the same  estimate holds for $(F^\la)^{-1}$.

\vspace{1mm}
We need to strengthen the distance estimate \re{uniform-dist} by the following uniform version of the Hopf Lemma:

\begin{lemma}\label{hopf-lem}
There exist $\varepsilon > 0$, $C'>0$, and $\la_0>0$ such that for $0 \leq \lambda \leq \lambda_0$ and every $z_{\lambda} \in D^{\lambda}$,
$$
\rho_{\la}(z_{\lambda}) \leq
  C'd(z_{\lambda},\partial D^{\lambda})\sup\left\{{\rho_{\la}}(z')\colon z'\in D^\la,\dist(z',\pd D^\la)\geq\e\right\}
$$
holds for any negative plurisubharmonic function $\rho_{\la}$ in $D^\la$, which is continuous on $\overline{D}^{\la}$.
\end{lemma}
\begin{proof}
Here $\e$ is a fixed positive constant so that for any $p\in\pd D^\la$ there is a ball $B_p(p_{\e})$ of radius $\e$ centered at $p_{\e}$ so that $\ov {B_{\e}(p_\e)}\cap \pd D^\la=p$ and $B_{\e}(p_{\e})\subset D^\la$. The existence of such a constant is given by Condition (ii)  in the proof of \rl{claim}. We point out that the classical Hopf Lemma does not impose the continuity of the function up to the boundary. However, to avoid technical adjustments, we assume the functions to be continuous up to the boundaries  and this special case suffices our application of Lemma~\ref{hopf-lem}.

	The proof of Lemma~\ref{hopf-lem} is classical for one single domain and a given plurisubharmonic function $\rho$ (see, for instance, \cite{FoSt}*{p.~58}). The proof is reduced to the inequality $\rho(z) \leq  C'd(z,\partial D)$. We repeat the argument of \cite{FoSt} showing how to obtain a uniform bound on $\la$.

The uniformity with respect to $\lambda$ relies on the following fact. For $p \in \partial D$, where $D \subset \subset \mathbb C^n$ with $\partial D$ of class $  C^2$, we denote by $\vec{n}(p)$ the unit exterior normal vector to $\partial D$ at $p$.  Set $L(p):=p + \mathbb C \vec{n}(p)$ and, for $\varepsilon > 0$, we denote by $p_{\varepsilon}$ the unique point in $D \cap (p+\mathbb R \vec{n}(p))$ such that $\|p-p_{\varepsilon}\| = \varepsilon$. Then, from Condition (ii) in the proof of \rl{claim} and changing $\lambda_0$ if necessary, there exists $\e>0$ so that, for every $0\leq\la\leq\la_0$ and  every $p^{\lambda} \in \partial D^{\lambda}$,
\begin{equation}\label{eq3}
 \ov{\Delta_{p_{\e}^{\la}}}:=\{z \in L_{p_{\lambda}}\colon \|z-p_{\e}^{\lambda}\| \leq \e\} \subset D^{\lambda} \cup\{p^{\lambda}\}.
\end{equation}
\vspace{1mm}

The rest of the proof of Lemma~\ref{hopf-lem} follows now line by line the one for a single domain  (see \cite{FoSt}*{p.p.~57-59}), with the necessary adaptation using (\ref{eq3}). Indeed, let $\varphi_{p^\la_{\e}}$ be a harmonic extension of ${\rho_{\la}}|_{\partial  \Delta_{p_{\e}^{\la}}}$ and let $\tilde{\varphi}_{p^\la_{\e}}$ be a holomorphic function on $\Delta_{p_{\e}^{\la}}$ with
$$
\RE(\tilde{\varphi}_{p^\la_{\e}})=\varphi_{p^\la_{\e}}.
 $$
 Then  $\varphi_{p^\la_{\e}}(p^{\la}_{\e}) = \frac{1}{2\pi}\int_0^{2\pi}\rho_{\la}(p^{\la}_{\e}+\e e^{it})dt$.
Moreover, from the continuity of the family $\mathcal D$, we have for  each small and positive $\e$
$$
\inf_{0 \leq \la \leq \la_0}Area\left(D^{\la}_{\e}:=\{z_\la \in \Delta_{p^\la_\e}\colon \dist(z_\la,\partial D^{\la}) \geq \e\}\right) = \mu_{\e} > 0.
$$
Hence
\begin{equation}\label{unif-eq}
\varphi_{p^\la_{\e}}(p^{\la}_{\e}) \leq
\frac{\mu_{\e}}{ \pi\e^2} \sup_{D^{\la}_{\e}}\varphi_{p^\la_{\e}} =:-C_{\la,\e} <0,
\end{equation}
which implies that $|a_{\la,\e}| \leq e^{-C_{\la,\e}}$ for $a_{\la,\e}:= e^{\tilde{\varphi}_{p^\la_{\e}}(p^\la_\e)}$.

Moreover, following \cite{FoSt} and using the Schwarz Lemma, the map
$$
g_{p^\la_{\e}}: \Delta_{p_{\e}^{\la}}\ni z   \mapsto \frac{e^{\tilde{\varphi}_{p^\la_{\e}}(z)}-a_{\la,\e}}{1-\ov{a_{\la,\e}}e^{\tilde{\varphi}_{p^\la_{\e}}(z)}} \in \Delta
$$
satisfies the inequality:
$$
\forall\,  0 < \delta < \e,\ |g_{p^\la_{\e}}(p^\la_\e+\delta \vec{n}(p^\la))| \leq \frac{\delta}{\e}.
$$
We now want to estimate $\left| e^{\tilde{\varphi}_{p^\la_{\e}}(p^\la_{\e}+\delta \vec{n}(p^\la))} \right|$. It suffices to estimate the largest modulus $r$ of $w$ satisfying
$$
\left|\frac{w-a_{\la,\e}}{1-\ov{a_{\la,\e}}w}\right|\leq\f{\del}{\e}.
$$
We have
$$
r \leq\frac{\f{\delta}{\e}+ |a_{\la,\e}| }{1+\f{\delta}{\e}|a_{\la,\e}|}\leq  \frac{\f{\delta}{\e}+ e^{-C_{\la,\e}} }{1+\f{\delta}{\e}e^{-C_{\la,\e}}}=:\eta_{\la,\e}.
$$
By the mean-value-theorem, we obtain $\eta_{\la,\e}\leq1-\f{1-\e^{-2}\del^2}{4}C_{\la,\e}\leq 1-\f{1-\e^{-1}\del}{4}C_{\la,\e}$.
Hence
\aln
\forall\,  0 < \delta < \e,\ e^{\varphi_{p^\la_\e}(p^\la_{\e}+\delta \vec{n}(p^\la))}& =\left| e^{\tilde{\varphi}_{p^\la_{\e}}(p^\la_{\e}+\delta \vec{n}(p^\la))} \right| \leq1-\f{1-\e^{-1}\del}{4}C_{\la,\e}\\
&
= 1-\frac{C_{\la,\e}}{4\e}\dist\left(p^\la_\e+\delta \vec{n}(p^\la),\partial D^{\la}\right).
\end{align*}
This finally leads to the conclusion:
$$
\forall\,  0 < \delta < \e,\ \varphi_{p^\la_\e}(p^\la_{\e}+\delta \vec{n}(p^\la)) \leq -\frac{C_{\la,\e}}{4\e}\dist\left(p^\la_\e+\delta \vec{n}(p^\la),\partial D^{\la}\right).
$$
This completes the proof of Lemma~\ref{hopf-lem}, using Estimate (\ref{unif-eq}). \end{proof}

\vspace{2mm}
We recall that $r^\la$ (resp.~$\tilde r^\la$) converges to $r^0$ (resp.~$\tilde r^0$) when $\la$ tends to zero. Hence, applying \re{uniform-dist} and the Hopf lemma to $\tilde r^\la\circ F^\la$ and $r^\la\circ (F^\la)^{-1}$, we obtain the following: There are $\la_0>0$ and constants $C,C'$ independent of $\la$ so that $\forall\,  0 \leq \la \leq \la_0$:
\eq{exact-dist}
C'\dist(z^\la,\pd D^\la)\leq \dist(F^\la(z^\la),\pd\tilde D^\la)\leq C\dist(z^\la,\pd D^\la).
\eeq

\vspace{2mm}
The proof of Theorem~\ref{boundary} also relies on uniform estimates of the Kobayashi infinitesimal metric. If $M$ is a complex manifold, we denote by $k_M$ the Kobayashi infinitesimal (pseudo)metric on $M$. We recall that by definition,  for $z \in M,\ v \in T_zM$,
$$
 k_M(z,v) = \inf\{\alpha >0\mid \exists f: \Delta \xrightarrow{\rm hol.} M,\ f(0) = z,\ f'(0) = v/\alpha\}.
$$

We have the following:

\begin{lemma}\label{kob-lem}
There exist $c,\la_0$, with $0 < c < 1$ and $0 < \la_0 < 1$, satisfying the following.
\begin{itemize}
\item[(i)] $\ \forall\,  0 \leq \la \leq \la_0,\ \forall\,  (z_{\lambda},v) \in D^{\lambda} \times \mathbb C^n:$

$$
 \frac{c\|v\|}{\sqrt{\dist(z_{\lambda},\partial D^{\lambda})}} \leq k_{D^{\lambda}}(z_{\lambda},v) \leq \frac{\|v\|}{\dist(z_{\lambda},\partial D^{\lambda})}.
$$

\item[(ii)] $\ \forall\,  \ 0 \leq \la \leq \la_0,\ \forall\,  (\tilde z_{\lambda},v) \in \tilde{D}^{\lambda} \times \mathbb C^n:$
$$
\frac{c\|v\|}{\sqrt{\dist(\tilde z_{\lambda},\partial \tilde{D}^{\lambda})}} \leq k_{\tilde{D}^{\lambda}}(z_{\lambda},v) \leq \frac{\|v\|}{\dist(\tilde z_{\lambda},\partial \tilde{D}^{\lambda})}.
$$
\end{itemize}
\end{lemma}

\begin{proof}
The proof  is standard and is due to Graham~\cite{Gra}  in the case of a single domain. For the convenience of the reader, we give a sketch of the proof showing how to obtain uniformity.
\begin{itemize}
\item[--]The upper estimates only use the fact that for a point $p$ contained in a domain $D \subset \subset \mathbb C^n$, the Euclidean ball $B(p,\dist(p,\partial D))$ is contained in $D$. Thus no extra argument is needed for the parameter version.
\item[--]  For the lower estimates, we provide some details.
\end{itemize}

It follows from condition (ii) in the proof of \rl{claim} that there exist positive $ \la_0$ and  $\varepsilon > 0$ such that for every $0 \leq \lambda \leq \lambda_0$ and for every $q_{\lambda} \in \partial D^{\lambda}$:
\begin{equation}\label{eq1}
{\rm the \ map} \ z \in D^{\lambda} \mapsto r(z) - \varepsilon \|z-q_{\lambda}\|^2 \ {\rm is \ strictly \ plurisubharmonic.}
 \end{equation}

 \vspace{1mm}
Finally, still using (ii) and shrinking $\lambda_0$ if necessary, there exists $c_1 > 0$ such that for every $0 \leq \lambda \leq \lambda_0$ and for every $z_{\lambda} \in D^{\lambda}$ we have:

 \begin{equation}\label{eq2}
 r_{\lambda}(z_{\lambda}) \geq -c_1  \dist(z_{\lambda},\partial D^{\lambda}).
 \end{equation}

\vspace{1mm}
Now, the proof of the above-mentioned lower estimates for the Kobayashi infinitesimal metric, only using (\ref{eq1}), (\ref{eq2}), the mean-value inequality,  and H\"older's inequality,   is  standard (see, for instance, \cite{FoSt} p.p. 56-57). This proves Lemma~\ref{kob-lem}.
\end{proof}

\vspace{1mm}
As a direct consequence of Lemmas~\ref{hopf-lem}, \ref{kob-lem} and  Estimate \re{exact-dist}, we obtain (see \cite{FoSt}*{p.~61}):
\begin{prop}\label{hold-prop} There exist $0<\la_0<1$ and $C>1$ so that
$$
\forall\, 0 \leq \la \leq \la_0,\ \forall\,  z_{\lambda} \in D^{\lambda},\ \forall\,  v \in \mathbb C^n,\ \|d_{z_{\lambda}}F^{\lambda}(v)\| \leq C \frac{\|v\|}{\sqrt{\dist(z_{\lambda},\partial D^{\lambda})}}.
$$
\end{prop}
Finally, using the Hardy-Littlewood Lemma (see \cite{FoSt}*{p.p.~62-63}), we have:
\begin{prop}\label{ext-prop}
Under the assumptions of \rta{boundary}, there exist positive constants $ \la_0$ and  $C$ such that for all $0 \leq \lambda \leq \lambda_0$ and  $z_{\lambda},w_{\lambda} \in D^{\lambda}:$
$$
\|F^{\lambda}(z_{\lambda})-F^{\lambda}(w_{\lambda})\| \leq C \|z_{\lambda}-w_{\lambda}\|^{1/2}.
$$
\end{prop}

\medskip

\begin{proof}[{\bf  Proof of \rta{boundary}}] The proof of Theorem~\ref{boundary} is now a direct consequence of Proposition~\ref{ext-prop} and
\rl{claim}.
\end{proof}

Our main positive result, motivated by the counter-examples, allows us to localize the study of regularity of families of biholomorphisms.
\begin{cor}\label{cor47}
Under the assumptions of \rta{boundary}, if $\cL D$, $\widetilde{\cL D}$ are continuous families of $\mathcal C^{\infty}$ bounded strictly pseudoconvex domains in   $\cc^n$ with $n\geq2$, then the map $x \mapsto F^\la\circ \Gaa^\la(x)-F^0\circ\Gaa^0(x)$ and its partial derivatives in $x$ of any order  converge uniformly to $0$ in $\ov D$ as $\la\to0$ for any parameterization of $\cL D$
with $\Gaa\in C^{\infty,0}(\ov{\cL D})$.
\end{cor}
\begin{proof}
Fix $p_0\in\pd D^0$. Let $q_0=F^0(p_0) \in   \partial \tilde D^0$. Without loss of generality we may assume that there is a strictly convex $C^\infty$ domain $\Om_0\subset D^0$  so that $p_0\in\pd\Om^0\cap\pd D^0$. Let $\Om^\la=\Gaa^\la\circ(\Gaa^0)^{-1}(\Om^0)$. Since $\Gaa$ is of class $C^{\infty,0}$, we may further assume that all $\Om^\la$ are strictly convex with $C^\infty$ boundary when $\la$ is sufficiently small. Since $\{(F^\la)^{-1}\}$ is also in $C^{0,0}(\ov D\times\{0\})$, we can find strictly convex domains $\om_1^\la$  contained in $F^\la(\Om^\la)$ so that  $F^\la(p^\la_0)\in\pd\om_1^\la\cap   \partial \tilde D^\la$. Here   we chose a strongly convex domain $\om_1^0$ contained in $F^0(\Om^0)$ and such that $F^0(p_0)\in\pd\om_1^0 \cap \partial \tilde D^0$. Then we set  $p_0^\la=\Gaa^\la\circ(\Gaa^0)^{-1}(p_0)$ and  $\om_1^\la=\tilde\Gaa^\la\circ(\tilde\Gaa^0)^{-1}(\om_1^0)$.

According to a result of
Lempert~\ci{Le81}*{p.p. 467-468},
we can find a stationary disc $\gaa_0$ for the strictly convex domain $\om^0$ such that $\pd\gaa_0\subset\pd D^0\cap\pd\om^0$ and $p_0\in\gaa_0$. Furthermore,  we may choose $\gaa_0$ that is contained in any fixed small neighborhood of $p_0$. Since $F\in C^{0,0}(\ov D\times\{0\})$,
then
\eq{fg0}
F^0(\gaa_0)\subset{\om_1^0}\cup (\pd\om_1^0\cap\pd \tilde D^0).
\eeq
 Since $F^0(\gaa_0)\setminus\pd F^0(\gaa_0)$ is a Kobayashi extremal disc in $F^0(\Om^0)$, then by \re{fg0}, it is also a Kobayashi extremal disc in $\om_1^0$ by the distance decreasing property of the inclusion map.
 Fix $c^0\in\gaa_0$. Let $c^\la=c^0$ and $\tilde c^\la=F^\la(c^\la)$. There exist  Riemann mappings
 $$
 R^\la\colon   \mathbb B^n  \to \Om^\la,\quad R_1^\la\colon   \mathbb B^n  \to \om_1^\la
 $$
 satisfying $R^\la(0)=c^\la$ and $R_1^\la(0)=\tilde c^\la$. Here  $\mathbb B^n$ denotes the unit ball in $\mathbb C^n$.  Without loss of generality, we may assume that $F^\la$ is tangent to the identity at $c^\la$ for every $\la$.
 By the uniqueness of Kobayashi extremal discs for bounded strictly convex $C^2$  domains, we have
 $$
 (R_1^0)^{-1}F^0 (\gaa_0)=(R^0)^{-1}\gaa_0.
 $$
 By the continuous dependence of stationary discs, we have,   for every $\la$,
 \eq{R1La}
 (R_1^\la)^{-1}F^\la (\gaa_0^\la)=(R^\la)^{-1}(\gaa_0^\la),
 \eeq
 where each $\gaa^\la_0$ is a stationary curve of $\Om^\la$ tangent to $\gaa_0^0$ at $c^\la$.  We can replace  $\gaa_0$  by any stationary disc  $\tilde \gaa_0$ such that the tangent line of $\tilde\gaa_0$ is a small perturbation of the tangent line of $\gaa_0$ at $c^0$, while \re{R1La} remains valid. This shows that, near $p_0^\la$,
 \begin{equation*}
 F^\la=R_1^\la\circ (R^\la)^{-1}.
 \end{equation*}
  Using Proposition 10 of \ci{Le81}, we know that each $R^\la$ is in $C^\infty(\ov{\mathbb B^n})$ and also $R^\la\colon\ov{\mathbb B^n}\to\ov{ D^\la}$ depends on $\la$ continuously. This shows that $\{F^\la\circ\Gaa^\la\}$ is in $C^{\infty,0}(\ov{\cL D})$. \qedhere
\end{proof}

As another application of \rp{hold-prop} and \rp{ext-prop} and the above proof, we have the following extension result.
\begin{cor}\label{cor48}
Let $n\geq2$. Let $\cL D,\widetilde{\cL D}$ be $C^\infty$ families of $C^{\infty}$ $($resp.~$C^\om)$ bounded strictly pseudoconvex domains in $\cc^n$. Let $F$ be a family of biholomorphic mappings $F^{\la}$ from $D^\la$ to $\tilde D^\la$. Suppose that $F\in C^{\infty,\infty}(\cL D)$ $($resp.~$C^\om)$. Then
$F\in C^{\infty,\infty}(\ov{\cL D})$ $($resp.~$C^{\om,\om}(\ov{\cL D}))$.
\end{cor}

\newcommand{\doi}[1]{\href{http://dx.doi.org/#1}{#1}}
\newcommand{\arxiv}[1]{\href{https://arxiv.org/pdf/#1}{arXiv:#1}}

  \def\MR#1{\relax\ifhmode\unskip\spacefactor3000 \space\fi%
  \href{http://www.ams.org/mathscinet-getitem?mr=#1}{MR#1}}

\end{document}